\documentclass[11pt,reqno]{amsart}

\usepackage{setspace}
\usepackage[margin=1in]{geometry}
\usepackage[hang,flushmargin,symbol*]{footmisc}
\usepackage{amsmath}
\usepackage{amsthm}
\usepackage{amssymb}
\usepackage{mathtools}
\usepackage{enumitem}
\usepackage{calc}
\usepackage{graphicx}
\usepackage{caption}
\usepackage[labelformat=simple,labelfont={}]{subcaption}
\usepackage{tikz}
\usetikzlibrary{decorations.markings}
\usetikzlibrary{arrows,shapes,positioning}
\usepackage{url}
\usepackage{array}

\theoremstyle{definition}
\newtheorem{theorem}{Theorem}[section]
\newtheorem{lemma}[theorem]{Lemma}

\newtheorem{corollary}[theorem]{Corollary}
\newtheorem{proposition}[theorem]{Proposition}

\newtheorem{example}[theorem]{Example}
\newtheorem{remark}[theorem]{Remark}

\newcommand{\Z}{\mathbb{Z}}
\newcommand{\N}{\mathbb{N}}
\DeclareMathOperator{\TL}{TL}
\DeclareMathOperator{\DTL}{\mathbb{D}TL}
\DeclareMathOperator{\depth}{depth}
\renewcommand{\P}{\mathcal{P}}
\providecommand{\abs}[1]{\left\lvert#1\right\rvert}
\newcommand{\w}{\mathsf{w}}
\renewcommand{\H}{\mathcal{H}}
\DeclareMathOperator{\FC}{FC}

\definecolor{naugreen}{cmyk}{.43,0,.34,.38}
\definecolor{naublue}{cmyk}{1,.72,0,.32}
\definecolor{mediterranean}{cmyk}{.67,0,.08,.3}
\definecolor{rose}{cmyk}{0,1.00,.20,0}
\definecolor{darkorchid}{cmyk}{.6,.9,0,.05}
\definecolor{butterfly}{cmyk}{.95,.59,0,.10}
\definecolor{springgreen}{cmyk}{1.00,0,.70,.02}
\definecolor{darkred}{cmyk}{0,1,1,.5}
\definecolor{nectarine}{cmyk}{0,0.70,1.00,0}
\definecolor{icyblue}{cmyk}{.84,.25,0,.06}
\definecolor{orange}{RGB}{255,102,0}
\definecolor{ggreen}{RGB}{0,153,0}
\definecolor{darkblue}{RGB}{0,0,255}
\definecolor{purple}{RGB}{153,51,255}
\definecolor{turq}{RGB}{72,209,204}

\newcommand\xxaxis{0}
\newcommand\yyaxis{90}

\newcommand\heapblock[4]{\fill[fill=#4, fill opacity=0.25, draw=#4, line width=1.1pt, rounded corners,shift={(\xxaxis:#1)},shift={(\yyaxis:#2)}] (-1,-0.5) rectangle (1,0.5);\node at (#1,#2) {\footnotesize $#3$};}

\newcommand\lp[1]{
\fill[fill=white, draw=black, xshift=-1, shift={(\yyaxis:#1)}] (0.75,#1) ellipse (16pt and 8pt);
}

\begin{document}

\title{Factorization of Temperley--Lieb diagrams}

\author[Ernst]{Dana C.~Ernst}
\address{Department of Mathematics and Statistics, Northern Arizona University, Flagstaff, AZ 86011}
\email{dana.ernst@nau.edu}
\urladdr{http://dcernst.github.io}

\author[Hastings]{Michael G.~Hastings}
\address{Department of Mathematics and Statistics, Northern Arizona University, Flagstaff, AZ 86011}
\email{mgh64@nau.edu}

\author[Salmon]{Sarah K.~Salmon}
\address{Department of Mathematics, University of Colorado Boulder, Boulder, CO 80309}
\email{sarah.salmon@colorado.edu}

\date{\today}


\maketitle


\begin{abstract}
The Temperley--Lieb algebra is a finite dimensional associative algebra that arose in the context of statistical mechanics and occurs naturally as a quotient of the Hecke algebra arising from a Coxeter group of type $A$. It is often realized in terms of a certain diagram algebra, where every diagram can be written as a product of ``simple diagrams."  These factorizations correspond precisely to factorizations of the so-called fully commutative elements of the Coxeter group that index a particular basis. Given a reduced factorization of a fully commutative element, it is straightforward to construct the corresponding diagram. On the other hand, it is generally difficult to reconstruct the factorization given an arbitrary diagram. We present an efficient algorithm for obtaining a reduced factorization for a given diagram.
\end{abstract}


\section{Introduction}\label{sec:Intro}

The Temperley--Lieb algebra, invented by Temperley and Lieb in 1971~\cite{Temperley1971}, is a finite dimensional associative algebra that arose in the context of statistical mechanics. Penrose~\cite{Penrose1971} and Kauffman~\cite{Kauffman1990} showed that this algebra can be faithfully represented by a diagram algebra that has a basis given by certain diagrams. In 1987, Jones~\cite{Jones1999} showed that the Temperley--Lieb algebra occurs naturally as a quotient of the Hecke algebra arising from a Coxeter group of type $A$ (whose underlying group is the symmetric group). This realization of the Temperley--Lieb algebra as a Hecke algebra quotient was later generalized to the case of an arbitrary Coxeter group by Graham~\cite{Graham1995}. These generalized Temperley--Lieb algebras have a basis indexed by the fully commutative elements (in the sense of Stembridge~\cite{Stembridge1996}) of the underlying Coxeter group. In cases when diagrammatic representations are known to exist, it turns out that every diagram can be written as a product of ``simple diagrams."  Each factorization of a diagram corresponds precisely to a factorization of the fully commutative element that indexes the diagram. Given a diagrammatic representation and a reduced factorization of a fully commutative element, it is easy to construct the corresponding diagram. However, given an arbitrary basis diagram, it is generally difficult to reconstruct the factorization of the corresponding group element. In the (type $A$) Temperley--Lieb algebra, we have devised an algorithm for obtaining a reduced factorization for a given diagram.

This paper is organized as follows. In Section~\ref{sec:Preliminaries}, we recall the basic terminology of Coxeter groups, fully commutative elements, heaps, and the Temperley--Lieb algebra, as well as establish our notation and review several necessary results. Section~\ref{sec:TL-DiagramAlgebra} describes the construction of the diagram algebra that is a faithful representation of the Temperley--Lieb algebra. This section includes a description of both the so-called simple diagrams that generate the algebra, as well as the basis that is indexed by the fully commutative elements of the Coxeter group of type $A$.  We present our algorithm for factoring a given Temperley--Lieb diagram in terms of the heap associated to the corresponding fully commutative element in Section~\ref{sec:MainResults}.  We conclude with Section~\ref{sec:ClosingRemarks}, which details potential further research.



\section{Preliminaries}\label{sec:Preliminaries}

\subsection{Coxeter groups}\label{subsec:CoxeterGroups}
A \emph{Coxeter system} is a pair $(W,S)$ consisting of a finite set $S$ of generating involutions and a group $W$, called a \emph{Coxeter group}, with presentation
\[
W = \langle S \mid (st)^{m(s,t)} = e ~\text{for}~ m(s,t) < \infty\rangle,
\]
where $e$ is the identity, $m(s,t) = 1$ if and only if $s = t$, and $m(s,t) = m(t,s)$.  It follows that the elements of $S$ are distinct as group elements and that $m(s,t)$ is the order of $st$~\cite{Humphreys1990}. Coxeter groups are generalizations of reflection groups, where each generator $s \in S$ can be thought of as a reflection. Recall that the composition of two reflections is a rotation by twice the angle between the corresponding hyperplanes. So if $s,t \in S$, we can think of $st$ as a rotation with order $m(s,t)$.

Since elements of $S$ have order two, the relation $(st)^{m(s,t)} = e$ can be written as
\begin{equation}\label{eqn:Braid}
\underbrace{sts \cdots}_{m(s,t)} = \underbrace{tst \cdots}_{m(s,t)} \end{equation}
with $m(s,t) \geq 2$ factors. If $m(s,t) = 2$, then $st = ts$ is called a \emph{commutation relation} since $s$ and $t$ commute. Otherwise, if $m(s,t) \geq 3$, then the corresponding relation is called a \emph{braid relation}. The replacement
\[
\underbrace{sts \cdots}_{m(s,t)} \mapsto \underbrace{tst \cdots}_{m(s,t)}
\]
will be referred to as a \emph{commutation} if $m(s,t) = 2$ and a \emph{braid move} if $m(s,t) \geq 3$.

We can represent the Coxeter system $(W,S)$ with a unique \emph{Coxeter graph} $\Gamma$ having
\begin{enumerate}
\item vertex set $S = \{s_1, \ldots, s_n\}$ and
\item edges $\{s_i,s_j\}$ for each $m(s_i,s_j) \geq 3$.
\end{enumerate}
Each edge $\{s_i,s_j\}$ is labeled with its corresponding bond strength $m(s_i,s_j)$. Since bond strength 3 is the most common, we typically omit the labels of 3 on those edges.

There is a one-to-one correspondence between Coxeter systems and Coxeter graphs. Given a Coxeter graph $\Gamma$, we can construct the corresponding Coxeter system $(W,S)$. In this case, we say that $(W,S)$, or just $W$, is of type $\Gamma$. If $(W,S)$ is of type $\Gamma$, for emphasis, we may write $(W,S)$ as $(W(\Gamma),S(\Gamma))$. Note that generators $s_i$ and $s_j$ are connected by an edge in the Coxeter graph $\Gamma$ if and only if $s_i$ and $s_j$ do not commute~\cite{Humphreys1990}.

The Coxeter system of type $A_n$ is given by the Coxeter graph in Figure~\ref{fig:A}. In this case, $W(A_n)$ is generated by $S(A_n) = \{s_1, s_2, \ldots, s_n\}$ and has defining relations
\begin{enumerate}
\item $s_is_i = e$ for all $i$;
\item $s_is_j = s_js_i$ when $\abs{i-j} > 1$;
\item $s_is_js_i = s_js_is_j$ when $\abs{i-j} = 1$.
\end{enumerate}
The Coxeter group $W(A_n)$ is isomorphic to the symmetric group $S_{n+1}$ under the mapping that sends $s_i$ to the adjacent transposition $(i,~i+1)$. This paper focuses on an associative algebra whose underlying structure is a Coxeter system of type $A_n$.

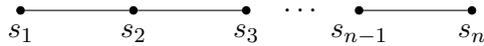
\begin{figure}[h!]
\begin{tabular}{m{7cm} m{7cm}}
\centering
\begin{tikzpicture}[scale=1,every circle node/.style={draw, circle, fill,inner sep=1pt}, node distance=1.5cm]
\node [circle, label=below:$s_1$] (s1) at (0,0){};
\node [circle, label=below:$s_2$, right of= s1] (s2){};
\node [circle, label=below:$s_3$, right of= s2] (s3){};
\node (dots)[right of=s3, node distance=.75cm]{$\cdots$};
\node [circle, label=below:$s_{n-1}$, right of=s3] (sn-1){};
\node [circle, label=below:$s_{n}$, right of=sn-1] (sn){};
\path[-] (s1) edge (s2);
\path[-] (s2) edge (s3);
\path[-] (sn-1) edge (sn);
\end{tikzpicture}
\end{tabular}
\caption{Coxeter graph of type $A_{n}$.} \label{fig:A}
\end{figure}

Let $S^*$ denote the free monoid over $S$. If a word $\w=s_{x_1}s_{x_2}\cdots s_{x_m}\in S^*$ is equal to $w$ when considered as an element of $W$, we say that $\w$ is an \emph{expression} for $w$. (Expressions will be written in {\sf sans serif} font for clarity.) Furthermore, if $m$ is minimal among all possible expressions for $w$, we say that $\w$ is a \emph{reduced expression} for $w$, and we call $m$ the \emph{length} of $w$, denoted $\ell(w)$.  Each element $w \in W$ can have several different reduced expressions that represent it. The following theorem, called Matsumoto's Theorem~\cite{Geck2000}, indicates how all of the reduced expressions for a given group element are related.

\begin{proposition}\label{thm:Matsumoto}
In a Coxeter group $W$, any two reduced expressions for the same group element differ by a finite sequence of commutations and braid moves. \qed
\end{proposition}

Let $\w$ be a reduced expression for $w \in W$. We define a \emph{subexpression} of $\w$ to be any subsequence of $\w$. We will refer to a consecutive subexpression of $\w$ as a \emph{subword}.

\begin{example}\label{ex:Subword}
Let $\w = s_1 s_2 s_4 s_5 s_2 s_6 s_5$ be an expression for $w \in W(A_6)$. Then we have
\[
s_1 \textcolor{blue}{s_2 s_4} s_5 s_2 s_6 s_5
= s_1 s_4 \textcolor{blue}{s_2 s_5} s_2 s_6 s_5
= s_1 s_4 s_5 \textcolor{ggreen}{s_2 s_2} s_6 s_5
= s_1 s_4 s_5 s_6 s_5,
\]
where the \textcolor{blue}{blue} subword indicates the location where a commutation is applied to obtain the next expression and the \textcolor{ggreen}{green} subword indicates the location where two adjacent occurrences of the same generator are canceled to obtain the last expression. This shows that $\w$ is not reduced. It turns out that $s_1 s_4 s_5 s_6 s_5$ is a reduced expression for $w$ and hence $\ell(w) = 5$.
\end{example}

\begin{example}
Let $\w = s_1s_2s_3s_4s_2$ be a reduced expression for $w\in W(A_4)$. Then the set of all reduced expressions for $w$ is given by
\[
\{s_1s_2s_3\textcolor{blue}{s_4s_2}, s_1\textcolor{rose}{s_2s_3s_2}s_4, \textcolor{blue}{s_1s_3}s_2s_3s_4, s_3s_1s_2s_3s_4\},
\]
where the \textcolor{blue}{blue} subwords indicate the location where a commutation is applied to obtain the next expression in the set and the \textcolor{rose}{pink} subword indicates the location where a braid move is applied to obtain the third expression from the second expression.  Note that $\ell(w) = 5$.
\end{example}


\subsection{Fully commutative elements}\label{subsec:FC}

Let $(W,S)$ be a Coxeter system of type $\Gamma$ and let $w \in W$. Following~\cite{Stembridge1996}, we define a relation $\sim$ on the set of reduced expressions for $w$. Let $\w$ and $\w'$ be two reduced expressions for $w$.  We define $\w \sim \w'$ if we can obtain $\w'$ from $\w$ by applying a single commutation move of the form $st \mapsto ts$, where $m(s,t) = 2$. Now, define the equivalence relation $\approx$ by taking the reflexive transitive closure of $\sim$. Each equivalence class under $\approx$ is called a \emph{commutation class}. Two reduced expressions are said to be \emph{commutation equivalent} if they are in the same commutation class.

\begin{example}\label{ex:nonFC}
Let $\w = s_1s_2s_3s_4s_5s_2$ and $\w' = s_1s_2s_3s_2s_4s_5$ be two different reduced expressions for $w\in W(A_5)$. Then $\w$ and $\w'$ are commutation equivalent since
\[
s_1s_2s_3s_4\textcolor{blue}{s_5s_2} = s_1s_2s_3\textcolor{blue}{s_4s_2}s_5 = s_1s_2s_3s_2s_4s_5,
\]
where the \textcolor{blue}{blue} subwords indicate the location where a commutation is applied to obtain the next expression. By applying a braid relation to $\w'$, we obtain
\[
s_1\textcolor{rose}{s_2s_3s_2}s_4s_5 = s_1s_3s_2s_3s_4s_5,
\]
where the location of the braid move has been highlighted in \textcolor{rose}{pink}. It turns out that the last reduced expression above is neither commutation equivalent to $\w$ nor $\w'$, and hence $w$ has more than one commutation class. Specifically, the commutation classes are
\[
\{s_1s_2s_3s_4s_5s_2, s_1s_2s_3s_4s_2s_5, s_1s_2s_3s_2s_4s_5\} ~\text{and}~ \{s_1s_3s_2s_3s_4s_5, s_3s_1s_2s_3s_4s_5\}.
\]
\end{example}

\begin{example}\label{ex:FC}
Let $\w = s_2s_1s_3s_4s_2$ be a reduced expression for $w\in W(A_4)$. In this case, $w$ has exactly five reduced expressions, including $\w$.  From this, it is easy to verify that all reduced expressions for $w$ are commutation equivalent. This implies that there is a unique commutation class for $w$:
\[
\{s_2s_1s_3s_4s_2, s_2s_3s_1s_4s_2, s_2s_1s_3s_2s_4, s_2s_3s_1s_2s_4, s_2s_3s_4s_1s_2\}.
\]
\end{example}

If $w$ has exactly one commutation class, then we say that $w$ is \emph{fully commutative}, or just FC. The set of all fully commutative elements of $W$ is denoted by $\FC(\Gamma)$, where $\Gamma$ is the corresponding Coxeter graph. For consistency, we say that a reduced expression $\w$ is FC if it is a reduced expression for some $w \in\FC(\Gamma)$. Note that the element in Example~\ref{ex:nonFC} is not FC since there are two commutation classes, while the element in Example~\ref{ex:FC} is FC.

Given some $w\in\FC(\Gamma)$ and a starting reduced expression for $w$, observe that the definition of fully commutative states that one only needs to perform commutations to obtain all the reduced expression for $w$, but the following result due to Stembridge~\cite{Stembridge1996} states that when $w$ is FC, performing commutations is the only possible way to obtain another reduced expression for $w$.

\begin{proposition}
\label{prop:Stembridge}
An element $w \in W$ is FC if and only if no reduced expression for $w$ contains $\underbrace{sts \cdots}_{m(s,t)}$ as a subword when $m(s,t) \geq 3$. \qed
\end{proposition}

In other words, an element is FC if and only if there is no opportunity to apply a braid move. For example, we can conclude that the element in Example~\ref{ex:nonFC} is not FC without actually computing the commutation classes since there is an opportunity to apply a braid move, which we highlighted in \textcolor{rose}{pink}.

Stembridge classified the irreducible Coxeter groups that contain only finitely many fully commutative elements, called the \emph{FC-finite Coxeter groups}. This paper is mainly concerned with $W(A_n)$, which is a finite group, so it has finitely many FC elements. However, there exist infinite Coxeter groups that contain only finitely many FC elements. For example, Coxeter groups of type $E_n$ with $n \geq 9$ are infinite, but they have only finitely many FC elements. It is well known that the number of FC elements in $W(A_{n})$ is equal to the $(n+1)$st Catalan number, where the $k$th Catalan number is given by $C_{k}=\frac{1}{k+1}\binom{2k}{k}$.


\subsection{Heaps}\label{subsec:Heaps}

Each reduced expression is associated with a labeled partially ordered set called a heap. Heaps provide a visual representation of the reduced expression while preserving the relations of the generators. We follow the development in~\cite{Ernst2010,Stembridge1996}.

Let $(W,S)$ be a Coxeter system. Suppose $\w = s_{x_1} s_{x_2} \cdots s_{x_k}$ is a reduced expression for $w \in W$, and as in~\cite{Stembridge1996}, define a partial ordering $\prec$ on the indices $\{1,\ldots,k\}$ by the transitive closure of the relation $j \prec i$ if $i < j$ and $s_{x_i}$ and $s_{x_j}$ do not commute. In particular, $j \prec i$ if $i < j$ and $s_{x_i} = s_{x_j}$, by transitivity and the fact that $\w$ is reduced. This partial order with $i$ labeled $s_{x_i}$ is called the \emph{heap} of $\w$. Note that for simplicity, we are omitting the labels of the underlying poset but retaining the labels of the corresponding generators.

\begin{example}\label{ex:FirstHeap}
Let $\w = s_2 s_1 s_3 s_2 s_4 s_5$ be a reduced expression for $w \in W(A_5)$.  Since $\ell(w)=6$, $\w$ is indexed by $\{1, 2, 3, 4, 5, 6\}$. We see that $4 \prec 3$ since $3 < 4$ and the third and fourth factors (namely, $s_3$ and $s_2$) do not commute. The labeled Hasse diagram for the heap of $\w$ is shown in Figure~\ref{fig:FirstHeapA}.
\end{example}

\begin{figure}[h!]
\centering
\begin{subfigure}[b]{0.3\linewidth}
\centering
\begin{tikzpicture}[scale=.55]
\draw (0,2)--(-1,1);
\draw (0,2)--(1,1);
\draw (-1,1)--(0,0);
\draw (0,0)--(1,1);
\draw (1,1)--(2,0);
\draw (2,0)--(3,-1);
\draw [fill=black] (0,2) circle (1.5pt);
\node[above] at (0,2) {$s_2$};
\draw [fill=black] (-1,1) circle (1.5pt);
\node[left] at (-1,1) {$s_1$};
\draw [fill=black] (1,1) circle (1.5pt);
\node[right] at (1,1) {$s_3$};
\draw [fill=black] (0,0) circle (1.5pt);
\node[below] at (0,0) {$s_2$};
\draw [fill=black] (2,0) circle (1.5pt);
\node[right] at (2,0) {$s_4$};
\draw [fill=black] (3,-1) circle (1.5pt);
\node[right] at (3,-1) {$s_5$};
\end{tikzpicture}
\caption{}
\label{fig:FirstHeapA}
\end{subfigure}
\begin{subfigure}[b]{0.3\linewidth}
\centering
\begin{tikzpicture}[scale=.55]
\heapblock{5}{1}{s_5}{black}
\heapblock{4}{2}{s_4}{black}
\heapblock{2}{2}{s_2}{black}
\heapblock{1}{3}{s_1}{black}
\heapblock{3}{3}{s_3}{black}
\heapblock{2}{4}{s_2}{black}
\end{tikzpicture}
\caption{}
\label{fig:FirstHeapB}
\end{subfigure}
\caption{Labeled Hasse diagram and lattice point representation of a heap.}
\label{fig:FirstHeap}
\end{figure}
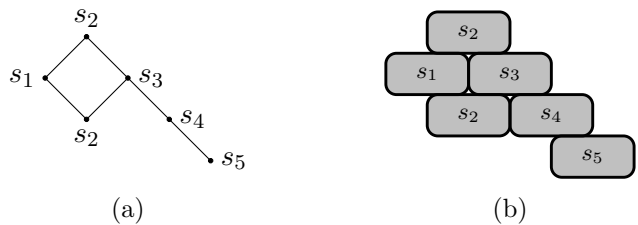

Let $\w$ be a fixed reduced expression for $w \in W(A_n)$. As in~\cite{Billey2007,Ernst2010}, we represent a heap for $\w$ as a set of lattice points embedded in $\{1,\ldots,n\} \times \N$. To do so, we assign coordinates $(x,y) \in \{1,\ldots,n\} \times \N$ to each entry of the labeled Hasse diagram for the heap of $\w$ in such a way that
\begin{enumerate}
\item An entry labeled $s_i$ in the heap has coordinates $(x,y)$ if and only if $x = i$;
\item An entry with coordinates $(x,y)$ is greater than an entry with coordinates $(x',y')$ in the heap if and only if $y > y'$.
\end{enumerate}
It follows from the definition that there is an edge in the Hasse diagram from $(x,y)$ to $(x',y')$ if and only if $x = x' \pm 1$, $y > y'$, and there are no entries $(x'', y'')$ such that $x'' \in \{x, x'\}$ and $y'< y'' < y$. This implies that we can completely reconstruct the edges of the Hasse diagram and the corresponding heap poset from a lattice point representation.

Note that our heaps are upside-down versions of the heaps that appear in~\cite{Billey2007} and several other papers. That is, in this paper, entries on top of a heap correspond to generators occurring to the left, as opposed to the right, in the corresponding reduced expression. One can form similar lattice point representations for heaps when $\Gamma$ is a straight line Coxeter graph.

Let $\w = s_{x_1} \cdots s_{x_k}$ be any reduced expression for $w \in W(A_{n})$. We let $H(\w)$ denote a lattice representation of the heap poset in $\{1,\ldots,n\} \times \N$ described in the paragraph above. There are many possible coordinate assignments for the entries of $H(\w)$, yet the $x$-coordinates for each entry will be fixed. If $s_{x_i}$ and $s_{x_j}$ are adjacent generators in the Coxeter graph with $i<j$, then we must place the point labeled by $s_{x_i}$ at a level that is \emph{above} the level of the point labeled by $s_{x_j}$. In particular, two entries labeled by the same generator may only differ by the amount of vertical space between them while maintaining their relative vertical position to adjacent entries in the heap.

Because generators that are not adjacent in the Coxeter graph commute, points whose $x$-coordinates differ by more than one can slide past each other or land at the same level. To visualize the labeled heap poset of a lattice representation we will enclose each entry of the heap in a block in such a way that if one entry covers another, the blocks overlap halfway.

It follows from~\cite[Proposition 2.2]{Stembridge1996} that heaps are well-defined up to commutation class.  In particular, there is a one-to-one correspondence between commutation classes and heaps.  That is, if $\w$ and $\w'$ are two reduced expressions for $w \in W$ that are in the same commutation class, then the heaps of $\w$ and $\w'$ are equal. Conversely, if $\w$ and $\w'$ belong to different commutation classes, then the corresponding heaps will be different.  In particular, if $w$ is FC, then it has a single commutation class, and so there is a unique heap associated to $w$. In this case, if $w$ is FC, then we may write $H(w)$ to denote the heap of any reduced expression for $w$. We will not make a distinction between $H(w)$ and its lattice point representation.

There are potentially many different ways to represent a heap, each differing by the vertical placement of blocks. For example, we can place blocks in vertical positions that are as high as possible, as low as possible, or some combination of high and low. When $w$ is FC, we wish to make a canonical choice for the representation of $H(w)$ by giving all blocks at the top of the heap the same vertical position and placing all other blocks as high as possible. Note that our canonical representation of heaps of FC elements corresponds precisely to the unique heap factorization of~\cite[Lemma 2.9]{Viennot1986} and to the Cartier--Foata normal form for monomials~\cite{Cartier1969,Green2006a}.

\begin{example}
The canonical lattice point representation of $H(\w)$ for the reduced expression given in Example~\ref{ex:FirstHeap} is shown in Figure~\ref{fig:FirstHeapB}.
\end{example}

\begin{example}\label{ex:TwoHeaps}
Consider $w \in W(A_5)$ from Example~\ref{ex:nonFC}. This element has two commutation classes, and hence two heaps as given in Figure~\ref{fig:TwoHeaps}, where we have color-coded in \textcolor{rose}{pink} the blocks of each heap that correspond to the braid relation $s_2s_3s_2 = s_3s_2s_3$. Figure~\ref{fig:TwoHeapsA} corresponds to the commutation class $\{s_1s_2s_3s_4s_5s_2, s_1s_2s_3s_4s_2s_5, s_1s_2s_3s_2s_4s_5\}$ while Figure~\ref{fig:TwoHeapsB} corresponds to the commutation class $\{s_1s_3s_2s_3s_4s_5, s_3s_1s_2s_3s_4s_5\}$.
\end{example}

\begin{figure}[h]
\centering
\begin{subfigure}[b]{0.3\textwidth}
\centering
\begin{tikzpicture}[scale=0.55]
  \heapblock{5}{1}{s_5}{black};
  \heapblock{4}{2}{s_4}{black};
  \heapblock{1}{5}{s_1}{black};
  \heapblock{3}{3}{s_3}{rose};
  \heapblock{2}{2}{s_2}{rose};
  \heapblock{2}{4}{s_2}{rose};
\end{tikzpicture}
\caption{}
\label{fig:TwoHeapsA}
\end{subfigure}
\begin{subfigure}[b]{0.3\textwidth}
\centering
\begin{tikzpicture}[scale=0.55]
  \heapblock{5}{1}{s_5}{black};
  \heapblock{4}{2}{s_4}{black};
  \heapblock{1}{5}{s_1}{black};
  \heapblock{3}{3}{s_3}{rose};
  \heapblock{2}{4}{s_2}{rose};
  \heapblock{3}{5}{s_3}{rose};
\end{tikzpicture}
\caption{}
\label{fig:TwoHeapsB}
\end{subfigure}
\caption{Two different heaps corresponding to the same non-FC element.}
\label{fig:TwoHeaps}
\end{figure}
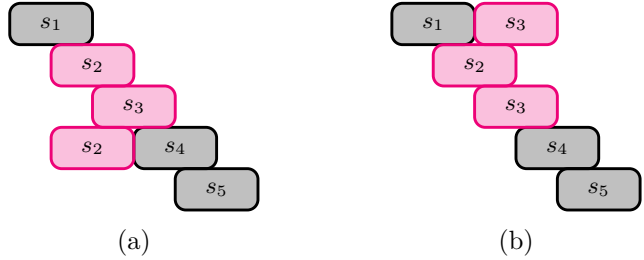

Given a heap, we can write a reduced expression for the corresponding group element by reading off the generators, starting at the top, moving left to right and then down. The expression we obtain is commutation equivalent to any expression to which the heap corresponds.


\subsection{The Temperley--Lieb algebra}\label{subsec:TL-algebra}

Given a Coxeter graph $\Gamma$, we can form the associative algebra $\TL(\Gamma)$, which we call the Temperley--Lieb algebra of type $\Gamma$~\cite{Graham1995}.  For a complete description of the construction of $\TL(\Gamma)$, see~\cite{Ernst2010,Graham1995,Green2006a}.  For our purposes, it suffices to define $\TL(A_{n})$ in terms of generators and relations.  We are using~\cite[Proposition~2.6]{Green2006a} (also see~\cite[Proposition~9.5]{Graham1995}) as our definition. 

The \emph{Temperley--Lieb algebra of type $A_{n}$}, denoted $\TL(A_{n})$, is the unital $\mathbb{Z}[\delta]$-algebra generated by $\{b_{1}, b_{2}, \dots, b_{n}\}$ with defining relations
\begin{enumerate}
\item $b_{i}^{2}=\delta b_{i}$ for all $i$;
\item $b_{i}b_{j}=b_{j}b_{i}$ if $|i-j|>1$;
\item $b_{i}b_{j}b_{i}=b_{i}$ if $|i-j|=1$.
\end{enumerate}

Suppose $w$ lies in $\FC(A_n)$ and has reduced expression $\w=s_{x_1}s_{x_2}\cdots s_{x_k}$.  Define the element $b_w \in\TL(A_n)$ via
\[
b_w=b_{x_1}b_{x_2}\cdots b_{x_k}.
\]
Notice that since $w$ is required to be fully commutative, the definition of $b_w$ is independent of choice of reduced expression for $w$.  It is well known (and follows from~\cite[Proposition 2.4]{Green2006a}) that the set $\{b_{w}\mid w \in \FC(A_n)\}$ forms a $\Z[\delta]$-basis for $\TL(A_n)$, called the \emph{monomial basis}.


\section{The Temperley--Lieb diagram algebra}\label{sec:TL-DiagramAlgebra}

Next, we establish our notation and introduce all of the terminology required to define an associative diagram algebra that is a faithful representation of $\TL(A_n)$.

Let $k$ be a nonnegative integer.  The \emph{standard $k$-box} is a rectangle with $2k$ points, called \emph{nodes}, labeled as in Figure~\ref{fig:k-box}.  We will refer to the top of the rectangle as the \emph{north face} and the bottom as the \emph{south face}.

\begin{figure}[!ht]
\centering
\begin{tikzpicture}[scale=.75]
\draw[gray,thick] (0,0) rectangle (8,2);
\foreach \x in {1,2,7} \filldraw (\x,0) circle (1pt);
\foreach \x in {1,2,7} \filldraw (\x,2) circle (1pt);
\draw (1,2) node[above]{\scriptsize $1$};
\draw (2,2) node[above]{\scriptsize $2$};
\draw (7,2) node[above]{\scriptsize $k$};
\draw (1,0) node[below]{\scriptsize $1'$};
\draw (2,0) node[below]{\scriptsize $2'$};
\draw (7,0) node[below]{\scriptsize $k'$};
\draw (4.25,2) node[above]{$\cdots$};
\draw (4.25,0) node[below]{$\cdots$};
\end{tikzpicture}
\caption{Standard $k$-box.}
\label{fig:k-box}
\end{figure}
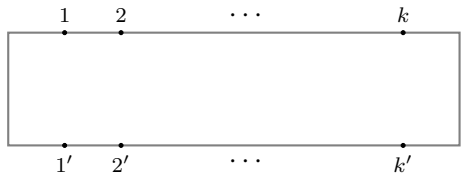

A \emph{concrete pseudo $k$-diagram} consists of a finite number of disjoint curves (planar), called \emph{edges}, embedded in the standard $k$-box with the following restrictions.
\begin{enumerate}
\item Every node of the box is the endpoint of exactly one edge, which meets the box transversely.
\item All other edges must be closed (isotopic to circles) and disjoint from the box.
\end{enumerate}

\begin{example}\label{ex:PseudoDiagram}
The diagram in Figure~\ref{fig:Psuedo} is an example of a concrete pseudo 6-diagram, whereas the diagram in Figure~\ref{fig:NotDiagram} does not represent a concrete pseudo 6-diagram since the diagram contains edges that are not disjoint (i.e., they intersect), node $4$ is the endpoint for more than one edge, nodes $3$ and $6'$ are not endpoints for any edge, and the edge leaving node $6$ does not have a node as its second endpoint.
\end{example}

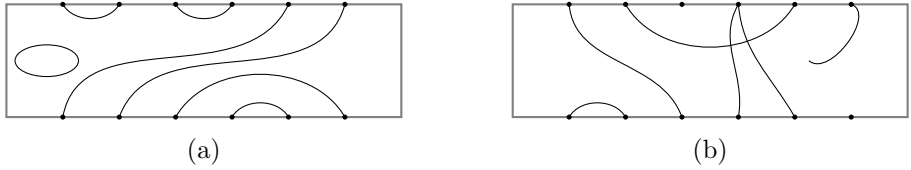
\begin{figure}[!h]
\centering
\begin{subfigure}[b]{0.4\linewidth}
\centering
\begin{tikzpicture}[scale=.75]
\draw[gray,thick] (0,0) rectangle (7,2);
\foreach \x in {1,2,3,4,5,6} \filldraw (\x,0) circle (1pt);
\foreach \x in {1,2,3,4,5,6} \filldraw (\x,2) circle (1pt);
\draw (3,0) to [bend left=60] (6,0);
\draw (2,0) to [out=70,in=-105] (6,2);
\draw (4,0) to [bend left=60] (5,0);
\draw (1,2) to [bend right=60] (2,2);
\draw (1,0) to [out=80,in=-115] (5,2);
\draw (3,2) to [bend right=60] (4,2);
\lp{0.5};
\end{tikzpicture}
\caption{}
\label{fig:Psuedo}
\end{subfigure}
\begin{subfigure}[b]{0.4\linewidth}
\centering
\begin{tikzpicture}[scale=.75]
\draw[gray,thick] (0,0) rectangle (7,2);
\foreach \x in {1,2,3,4,5,6} \filldraw (\x,0) circle (1pt);
\foreach \x in {1,2,3,4,5,6} \filldraw (\x,2) circle (1pt);
\draw (1,0) to [bend left=60] (2,0);
\draw (3,0) to [out=110,in=-85] (1,2);
\draw (4,0) to [out=80,in=-120] (4,2);
\draw (5,0) to [out=120,in=-85] (4,2);
\draw (2,2) to [bend right=60] (5,2);
\draw (6,2) to [out=-20,in=-40] (5.25,1);
\end{tikzpicture}
\caption{}
\label{fig:NotDiagram}
\end{subfigure}
\caption{Example of a concrete pseudo 6-diagram together with a non-example.}
\end{figure}

We now define an equivalence relation on the set of concrete pseudo $k$-diagrams. Two concrete pseudo $k$-diagrams are \emph{(isotopically) equivalent} if one concrete diagram can be obtained from the other by isotopically deforming the edges such that any intermediate diagram is also a concrete pseudo $k$-diagram. Note that an isotopy of the $k$-box is a 1-parameter family of homeomorphisms of the $k$-box to itself that are stationary on the boundary.

A \emph{pseudo $k$-diagram} is defined to be an equivalence class of equivalent concrete pseudo $k$-diagrams.  We denote the set of pseudo $k$-diagrams by $T_{k}$. Note that we used the word ``pseudo'' in our definition to emphasize that we allow loops to appear in our diagrams.

\begin{remark}\label{rem:VertEquiv}
When representing a pseudo $k$-diagram with a drawing, we pick an arbitrary concrete representative among a continuum of equivalent choices. When no confusion can arise, we will not make a distinction between a concrete pseudo $k$-diagram and the equivalence class that it represents. We say that two concrete pseudo $k$-diagrams are \emph{vertically equivalent} if they are equivalent in the above sense by an isotopy that preserves setwise each vertical cross-section of the $k$-box.
\end{remark}

\begin{example}
The concrete pseudo $5$-diagrams in Figure~\ref{fig:EquivDiagrams} are equivalent since each diagram can be obtained from the other by isotopically deforming the edges.
\end{example}

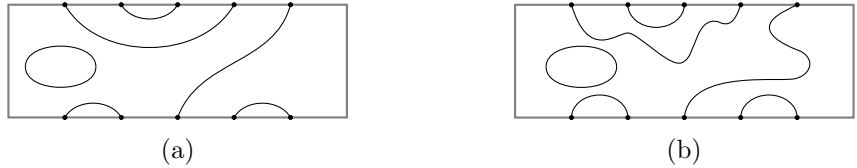
\begin{figure}[!h]
\centering
\begin{subfigure}[b]{0.4\linewidth}
\centering
\begin{tikzpicture}[scale=.75]
\draw[gray,thick] (0,0) rectangle (6,2);
\foreach \x in {1,2,3,4,5} \filldraw (\x,0) circle (1pt);
\foreach \x in {1,2,3,4,5} \filldraw (\x,2) circle (1pt);
\draw (1,0) to [bend left=60] (2,0);
\draw (3,0) to [out=70,in=-105] (5,2);
\draw (4,0) to [bend left=60] (5,0);
\draw (2,2) to [bend right=60] (3,2);
\draw (1,2) to [bend right=60] (4,2);
\draw (.3,.9) to [bend right=90] (1.55,.9);
\draw (.3,.9) to [bend left=90] (1.55,.9);
\end{tikzpicture}
\caption{}
\end{subfigure}
\begin{subfigure}[b]{0.4\linewidth}
\centering
\begin{tikzpicture}[scale=.75]
\draw[gray,thick] (0,-2) rectangle (6,0);
\foreach \x in {1,2,3,4,5} \filldraw (\x,0) circle (1pt);
\foreach \x in {1,2,3,4,5} \filldraw (\x,-2) circle (1pt);
\draw (1,0) .. controls (1.3,-1) and (1.8,-0.5) .. (2,-0.5) .. controls (2.2,-0.5) and (2.8,-1.2) .. (3,-1) .. controls (3.2,-0.9) and (3.2,-0.2) .. (3.6,-0.3) .. controls (3.8,-0.4) and (3.9,-0.3) .. (4,0);
\draw (2,0) arc (-180:0:0.5 and 0.4) ;
\draw (1,-2) arc (180:0:0.5 and 0.4) ;
\draw (4,-2) arc (180:0:0.5 and 0.4) ;
\draw (5,0) .. controls (4.8,-0.1) and (4,-0.4) .. (5,-0.8) .. controls (5.3,-0.9) and (5.3,-1.2) .. (5,-1.3) .. controls (4.8,-1.4) and (3.1,-1.1) .. (3,-2);
\draw (.55,-1.1) to [bend right=90] (1.8,-1.1);
\draw (.55,-1.1) to [bend left=90] (1.8,-1.1);
\end{tikzpicture}
\caption{}
\end{subfigure}
\caption{Isotopically equivalent diagrams.}
\label{fig:EquivDiagrams}
\end{figure}

Let $d$ be a diagram and let $e$ be an edge of $d$. If $e$ is a closed curve occurring in $d$, then we call $e$ a \emph{loop}.  For example, the diagram in Figure~\ref{fig:Psuedo} has a single loop.  If $e$ joins node $i$ in the north face to node $j'$ in the south face, then $e$ is called a \emph{propagating edge from $i$ to $j'$}. If $e$ is not propagating, loop or otherwise, it will be called \emph{non-propagating}. It is clear that there is a unique loop-free diagram consisting only of propagating edges.  This diagram, denoted by $d_0$, is depicted in Figure~\ref{fig:IdentityDiagram}.

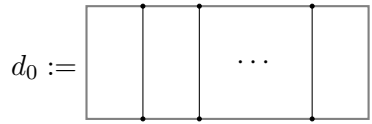
\begin{figure}[!ht]
\centering
\begin{align*}
d_{0} & := \begin{tabular}[c]{@{}l}\begin{tikzpicture}[scale=.75]
\draw[gray,thick] (0,0) rectangle (5,2);
\foreach \x in {1,2,4} \filldraw (\x,0) circle (1pt);
\foreach \x in {1,2,4} \filldraw (\x,2) circle (1pt);
\draw (3,1) node{$\cdots$};
\foreach \x in {1,2,4} \draw (\x,0) to (\x,2);
\end{tikzpicture}
\end{tabular}
\end{align*}
\caption{Unique loop-free diagram having only propagating edges.}\label{fig:IdentityDiagram}
\end{figure}

We wish to define an associative algebra that has the pseudo $k$-diagrams as a basis. Let $R$ be a commutative ring with $1$.  The associative algebra $\P_{k}$ over $R$ is the free $R$-module having $T_{k}$ as a basis. We define multiplication (referred to as diagram concatenation) in $\P_{k}$ by defining multiplication in the case where $d$ and $d'$ are basis elements, and then extending bilinearly. If $d, d' \in T_{k}$, the product $d'd$ is the element of $T_{k}$ obtained by placing $d'$ on top of $d$, so that node $i'$ of $d'$ coincides with node $i$ of $d$, and then removing the identified boundary to recover a standard $k$-box. If desired, one can then vertically rescale the resulting rectangle.

\begin{example}
Figure~\ref{fig:LoopDiagramMultiplication} depicts the product of two pseudo $5$-diagrams in $\P_5$.
\end{example}	

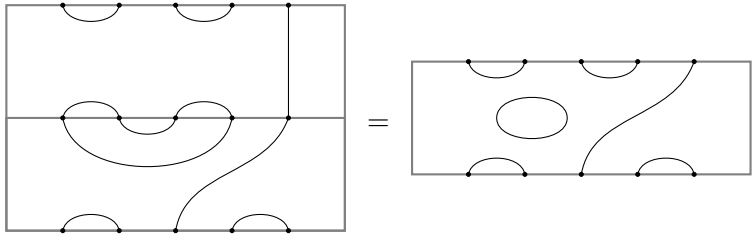
\begin{figure}[!ht]
\centering
\begin{align*}
\begin{tabular}[c]{@{}l}\begin{tikzpicture}[scale=.75]
\draw[gray,thick] (0,0) rectangle (6,2);
\draw[gray,thick] (0,0) rectangle (6,4);
\foreach \x in {1,2,3,4,5} \filldraw (\x,0) circle (1pt);
\foreach \x in {1,2,3,4,5} \filldraw (\x,2) circle (1pt);
\foreach \x in {1,2,3,4,5} \filldraw (\x,4) circle (1pt);
\draw (5,2) to (5,4);
\draw (3,0) to [out=80,in=250] (5,2);
\draw (1,0) to [bend left=80] (2,0);
\draw (4,0) to [bend left=80] (5,0);
\draw (1,2) to [bend left=80] (2,2);
\draw (3,2) to [bend left=80] (4,2);
\draw (1,2) to [bend right=80] (4,2);
\draw (2,2) to [bend right=80] (3,2);
\draw (1,4) to [bend right=80] (2,4);
\draw (3,4) to [bend right=80] (4,4);
\end{tikzpicture}\end{tabular}
& = \begin{tabular}[c]{l}\begin{tikzpicture}[scale=.75]
\draw[gray,thick] (0,0) rectangle (6,2);
\foreach \x in {1,2,3,4,5} \filldraw (\x,0) circle (1pt);
\foreach \x in {1,2,3,4,5} \filldraw (\x,2) circle (1pt);
\draw (1,0) to [bend left=80] (2,0);
\draw (3,0) to [out=80,in=250] (5,2);
\draw (4,0) to [bend left=80] (5,0);
\draw (1,2) to [bend right=80] (2,2);
\draw (3,2) to [bend right=80] (4,2);
\draw (1.5,1) to [bend right=90] (2.75,1);
\draw (1.5,1) to [bend left=90] (2.75,1);
\end{tikzpicture}
\end{tabular}
\end{align*}
\caption{Example of multiplication in $\P_5$.}\label{fig:LoopDiagramMultiplication}
\end{figure}

We now restrict our attention to the base ring $\Z[\delta]$, which is the ring of polynomials in $\delta$ with integer coefficients.  We define the \emph{Temperley--Lieb diagram algebra} $\DTL(A_{n})$ to be the associative $\Z[\delta]$-algebra equal to the quotient of $\P_{n+1}$ determined by the relation depicted in Figure~\ref{fig:LoopRelation}.

\begin{figure}[!ht]
\centering
\begin{tabular}{llc}
\begin{tikzpicture}[scale=.65]
\draw (2,0) to [bend right=90] (4,0);
\draw (2,0) to [bend left=90] (4,0);
\end{tikzpicture}
\end{tabular}
= $\delta$
\caption{Defining relation of $\DTL(A_{n})$.}\label{fig:LoopRelation}
\end{figure}

It is well known~\cite{Kauffman1987,Kauffman1990} that $\DTL(A_{n})$ is the free $\Z[\delta]$-module with basis given by the elements of $T_{n+1}$ having no loops. The multiplication is inherited from the multiplication on $\P_{n+1}$ except we multiply by a factor of $\delta$ for each resulting loop and then discard the loop. It is easy to see that the identity in $\DTL(A_{n})$ is the diagram $d_0$ given in Figure~\ref{fig:IdentityDiagram}. Technically, the identity diagram is the image of $d_0$ in the quotient algebra, but there is no danger of identifying the two diagrams.

\begin{example}\label{ex:DiagramMulitplication}
Figure~\ref{fig:DiagramMultiplication} depicts the product of three basis diagrams from $\DTL(A_{4})$.
\end{example}

\begin{figure}[!ht]
\centering
\begin{align*}
\begin{tabular}[c]{@{}l}\begin{tikzpicture}[scale=.75]
\draw[gray,thick] (0,0) rectangle (6,2);
\draw[gray,thick] (0,0) rectangle (6,4);
\draw[gray,thick] (0,0) rectangle (6,6);
\foreach \x in {1,2,3,4,5} \filldraw (\x,0) circle (1pt);
\foreach \x in {1,2,3,4,5} \filldraw (\x,2) circle (1pt);
\foreach \x in {1,2,3,4,5} \filldraw (\x,4) circle (1pt);
\foreach \x in {1,2,3,4,5} \filldraw (\x,6) circle (1pt);
\draw (5,2) to (5,4);
\draw (3,0) to [out=80,in=250] (5,2);
\draw (5,4) to [out=100,in=280] (3,6);
\draw (1,0) to [bend left=80] (2,0);
\draw (4,0) to [bend left=80] (5,0);
\draw (1,2) to [bend left=80] (2,2);
\draw (3,2) to [bend left=80] (4,2);
\draw (2,4) to [bend left=80] (3,4);
\draw (1,4) to [bend left=80] (4,4);
\draw (1,2) to [bend right=80] (4,2);
\draw (2,2) to [bend right=80] (3,2);
\draw (1,4) to [bend right=80] (4,4);
\draw (2,4) to [bend right=80] (3,4);
\draw (1,6) to [bend right=80] (2,6);
\draw (4,6) to [bend right=80] (5,6);
\end{tikzpicture}\end{tabular}
& = \delta^3 \begin{tabular}[c]{l}\begin{tikzpicture}[scale=.75]
\draw[gray,thick] (0,0) rectangle (6,2);
\foreach \x in {1,2,3,4,5} \filldraw (\x,0) circle (1pt);
\foreach \x in {1,2,3,4,5} \filldraw (\x,2) circle (1pt);
\draw (1,0) to [bend left=80] (2,0);
\draw (4,0) to [bend left=80] (5,0);
\draw (1,2) to [bend right=80] (2,2);
\draw (4,2) to [bend right=80] (5,2);
\draw (3,0) to (3,2);
\end{tikzpicture}
\end{tabular}
\end{align*}
\caption{Example of multiplication in $\DTL(A_{4})$.}\label{fig:DiagramMultiplication}
\end{figure}
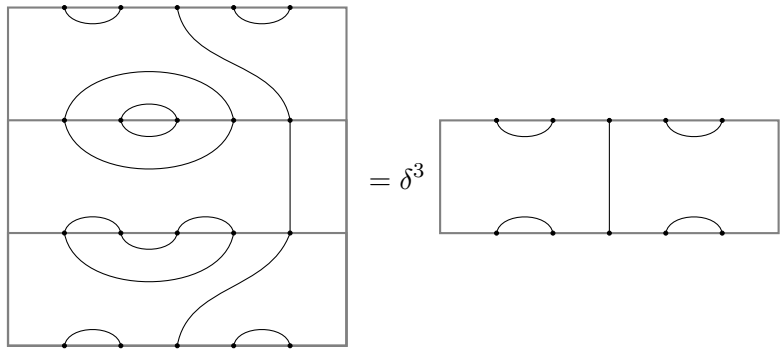

Define the \emph{simple diagrams} $d_{1}, d_{2}, \dots, d_{n}$ as in Figure~\ref{fig:SimpleDiagrams}.  Note that the simple diagrams are elements of the basis for $\DTL(A_{n})$.  It turns out~\cite{Kauffman1987,Kauffman1990} (and follows from Proposition~\ref{prop:TL=DTL} below) that the set of loop-free diagrams of $\DTL(A_n)$ is generated as a unital algebra by the set of simple diagrams $\{d_1,d_2,\ldots,d_n\}$.  In fact, $\DTL(A_n)$ is often defined to be the unital $\mathbb{Z}[\delta]$-algebra generated by the simple diagrams subject to the relation given in Figure~\ref{fig:LoopRelation}.

\begin{figure}[!ht]
\begin{center}
{\setlength\arraycolsep{1pt}
\begin{tabular}{@{}l@{}lc}
$d_{1}$ & $:=$ & \begin{tabular}[c]{@{}l}\begin{tikzpicture}[scale=.75]
\draw[gray,thick] (0,0) rectangle (9,2);
\foreach \x in {1,2,3,4,7,8} \filldraw (\x,0) circle (1pt);
\foreach \x in {1,2,3,4,7,8} \filldraw (\x,2) circle (1pt);
\draw (1,0) to [bend left=80] (2,0);
\draw (1,2) to [bend right=80] (2,2);
\draw (5.5,1) node{$\cdots$};
\foreach \x in {3,4,7,8} \draw (\x,0) to (\x,2);
\draw (1,2) node[above]{\scriptsize $\phantom{+}1\phantom{+}$};
\draw (2,2) node[above]{\scriptsize $\phantom{+}2\phantom{+}$};
\draw (7,2) node[above]{\scriptsize $\phantom{+}n\phantom{+}$};
\draw (8,2) node[above]{\scriptsize $n+1$};
\draw (8,0) node[below]{\phantom{\scriptsize $n+1$}};
\end{tikzpicture}\end{tabular}  \\
& & $\vdots$\vspace{0.5em} \\
$d_{i}$ & $:=$ & \begin{tabular}[c]{@{}l}\begin{tikzpicture}[scale=.75]
\draw[gray,thick] (0,0) rectangle (9,2);
\foreach \x in {1,3,4,5,6,8} \filldraw (\x,0) circle (1pt);
\foreach \x in {1,3,4,5,6,8} \filldraw (\x,2) circle (1pt);
\draw (4,0) to [bend left=80] (5,0);
\draw (4,2) to [bend right=80] (5,2);
\draw (2,1) node{$\cdots$};
\draw (7,1) node{$\cdots$};
\foreach \x in {1,3,6,8} \draw (\x,0) to (\x,2);
\draw (1,2) node[above]{\scriptsize $\phantom{+}1\phantom{+}$};
\draw (4,2) node[above]{\scriptsize $\phantom{+}i\phantom{+}$};
\draw (5,2) node[above]{\scriptsize $i+1$};
\draw (8,2) node[above]{\scriptsize $n+1$};
\draw (8,0) node[below]{\phantom{\scriptsize $n+1$}};
\end{tikzpicture}\end{tabular}  \\
& & $\vdots$ \\
$d_{n}$ & $:=$ & \begin{tabular}[c]{@{}l}\begin{tikzpicture}[scale=.75]
\draw[gray,thick] (0,0) rectangle (9,2);
\foreach \x in {1,2,5,6,7,8} \filldraw (\x,0) circle (1pt);
\foreach \x in {1,2,5,6,7,8} \filldraw (\x,2) circle (1pt);
\draw (7,0) to [bend left=80] (8,0);
\draw (7,2) to [bend right=80] (8,2);
\draw (3.5,1) node{$\cdots$};
\foreach \x in {1,2,5,6} \draw (\x,0) to (\x,2);
\draw (1,2) node[above]{\scriptsize $\phantom{+}1\phantom{+}$};
\draw (2,2) node[above]{\scriptsize $\phantom{+}2\phantom{+}$};
\draw (7,2) node[above]{\scriptsize $\phantom{+}n\phantom{+}$};
\draw (8,2) node[above]{\scriptsize $n+1$};
\draw (8,0) node[below]{\phantom{\scriptsize $n+1$}};
\end{tikzpicture}\end{tabular}
\end{tabular}}
\end{center}\caption{Simple diagrams.}\label{fig:SimpleDiagrams}
\end{figure}
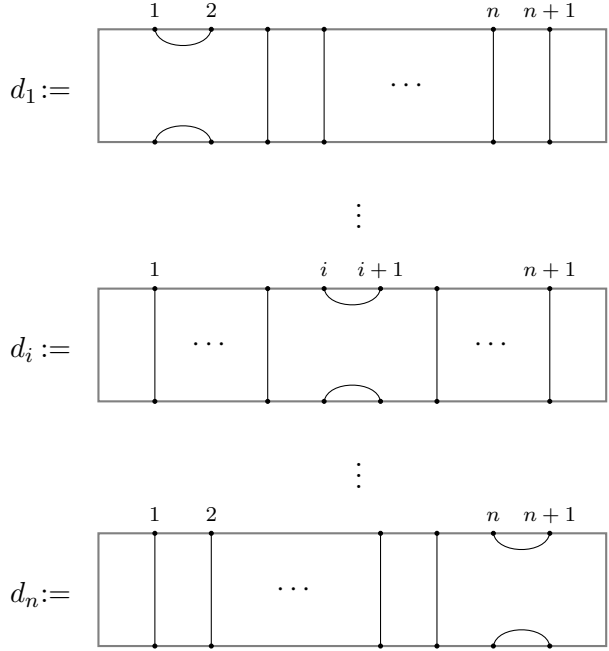

It is easy to verify that the simple diagrams of $\DTL(A_n)$ satisfy the defining relations of $\TL(A_n)$. That is, we have
\begin{enumerate}
\item $d_id_i = \delta d_i$ for all $i$;
\item $d_id_j = d_jd_i$ when $\abs{i-j} > 1$;
\item $d_id_jd_i = d_i$ when $\abs{i-j} = 1$.
\end{enumerate}
For example, Figure~\ref{fig:DiagramReln} illustrates the third relation above for the case $j=i+1$. Indeed, $\TL(A_n)$ and $\DTL(A_n)$ are isomorphic as $\mathbb{Z}[\delta]$-algebras (for instance, see~\cite[\S 3]{Kauffman1990}).

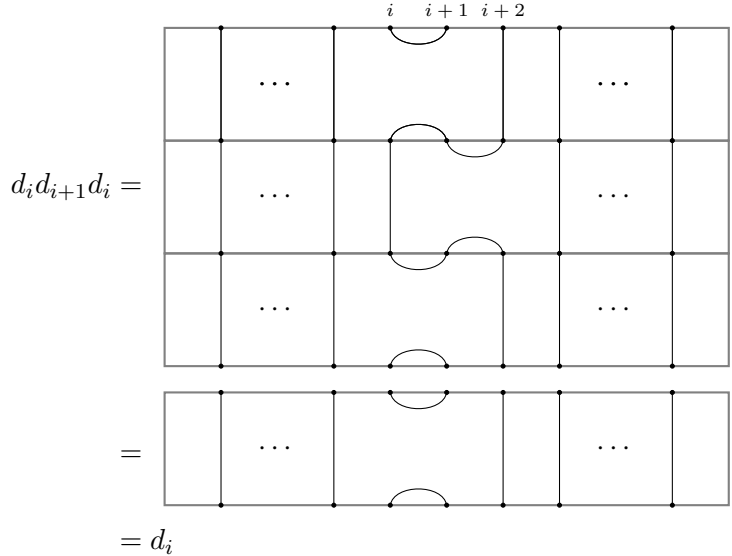
\begin{figure}[!ht]
\begin{align*}
d_i d_{i+1} d_i  &= 
\begin{tabular}[c]{l}
\begin{tikzpicture}[scale=.75]
\draw[gray,thick] (0,0) rectangle (10,2);
\draw[gray,thick] (0,2) rectangle (10,4);
\draw[gray,thick] (0,4) rectangle (10,6);
\foreach \x in {1,3,4,5,6,7,9} \filldraw (\x,0) circle (1pt);
\foreach \x in {1,3,4,5,6,7,9} \filldraw (\x,2) circle (1pt);
\foreach \x in {1,3,4,5,6,7,9} \filldraw (\x,4) circle (1pt);
\foreach \x in {1,3,4,5,6,7,9} \filldraw (\x,6) circle (1pt);
\draw (4,6) to [bend right=80] (5,6);
\draw (4,4) to [bend left=80] (5,4);
\foreach \x in {1,3,6,7,9} \draw (\x,4) to (\x,6);
\foreach \x in {2,8} \draw (\x,5) node{$\cdots$};
\draw (5,4) to [bend right=80] (6,4);
\draw (5,2) to [bend left=80] (6,2);
\foreach \x in {1,3,4,7,9} \draw (\x,2) to (\x,4);
\foreach \x in {2,8} \draw (\x,3) node{$\cdots$};
\draw (4,6) to [bend right=80] (5,6);
\draw (4,4) to [bend left=80] (5,4);
\foreach \x in {1,3,6,7,9} \draw (\x,4) to (\x,6);
\foreach \x in {2,8} \draw (\x,5) node{$\cdots$};
\draw (4,2) to [bend right=80] (5,2);
\draw (4,0) to [bend left=80] (5,0);
\foreach \x in {1,3,6,7,9} \draw (\x,0) to (\x,2);
\foreach \x in {2,8} \draw (\x,1) node{$\cdots$};
\draw (4,6) node[above]{\tiny $\phantom{+}i\phantom{+}$};
\draw (5,6) node[above]{\tiny $i+1$};
\draw (6,6) node[above]{\tiny $i+2$};
\end{tikzpicture}
\end{tabular}\\ 
& = \begin{tabular}[c]{l}
\begin{tikzpicture}[scale=.75]
\draw[gray,thick] (0,0) rectangle (10,2);
\foreach \x in {1,3,4,5,6,7,9} \filldraw (\x,0) circle (1pt);
\foreach \x in {1,3,4,5,6,7,9} \filldraw (\x,2) circle (1pt);
\draw (4,2) to [bend right=80] (5,2);
\draw (4,0) to [bend left=80] (5,0);
\foreach \x in {1,3,6,7,9} \draw (\x,0) to (\x,2);
\foreach \x in {2,8} \draw (\x,1) node{$\cdots$};
\end{tikzpicture}
\end{tabular}\\
&= d_i
\end{align*}
\caption{Special case of one of the relations in $\DTL(A_n)$.}\label{fig:DiagramReln}
\end{figure}

\begin{proposition}\label{prop:TL=DTL}
Let $\theta : \TL(A_n)\longrightarrow \DTL(A_{n})$ be the function determined by $b_{i}\longmapsto d_{i}$. Then $\theta$ is a well-defined $\mathbb{Z}[\delta]$-algebra isomorphism that maps the monomial basis of $\TL(A_{n})$ to the set of loop-free diagrams in $\DTL(A_{n})$. \qed
\end{proposition}

Let $w$ be an element of $\FC(A_n)$ and define $d_w$ to be the image of the monomial $b_w$. It follows that if $w$ has $s_{x_1} \cdots s_{x_k}$ as reduced expression, then $d_w = d_{x_1} \cdots d_{x_k}$. That is, given a reduced factorization for $w \in \FC(A_n)$, we can easily obtain a reduced factorization of $d_w$ in terms of simple diagrams. However, given a loop-free diagram $d$, it is more difficult to obtain a factorization. Resolving this difficulty is the content of Section~\ref{sec:MainResults}.

Let $\w=s_{x_1}\cdots s_{x_k}$ be a reduced expression for $w\in\FC(A_n)$. For each simple diagram $d_{x_i}$, fix a concrete representation such that the propagating edges are straight and the pair of non-propagating edges never double-back on themselves (i.e., the non-propagating edges never intersect any vertical line more than once). Now, consider the concrete diagram that results from concatenating the concrete simple diagrams $d_{x_1},\ldots,d_{x_k}$, rescaling vertically to recover the standard $(n+1)$-box, but not deforming any of the non-propagating edges. Since $w$ is $\FC$ and vertical equivalence respects commutation, given any two reduced expressions for $w$, the corresponding concrete diagrams constructed as above will be vertically equivalent (see Remark~\ref{rem:VertEquiv}). We define the corresponding vertical equivalence class to be the \emph{simple representation of} $d_w$. The simple representation of $d_w$ is designed to replicate the structure of the corresponding heap.

\begin{example}\label{ex:SimpleRepn}
Let $\w=s_{1}s_{3}s_{2}s_{4}s_{3}$ be a reduced expression for $w\in\FC(A_4)$.  The factorization for $d_w$ determined by $\w$ together with its simple representation is shown in Figures~\ref{fig:SimpleRepnA} and \ref{fig:SimpleRepnB}, respectively. The resulting product is $d_w$, which is shown in Figure~\ref{fig:SimpleRepnC}. The shaded regions in Figures~\ref{fig:SimpleRepnB} and \ref{fig:SimpleRepnC} indicate that the pair of edges bounding the top and bottom of the region arise from the same factor.
\end{example}

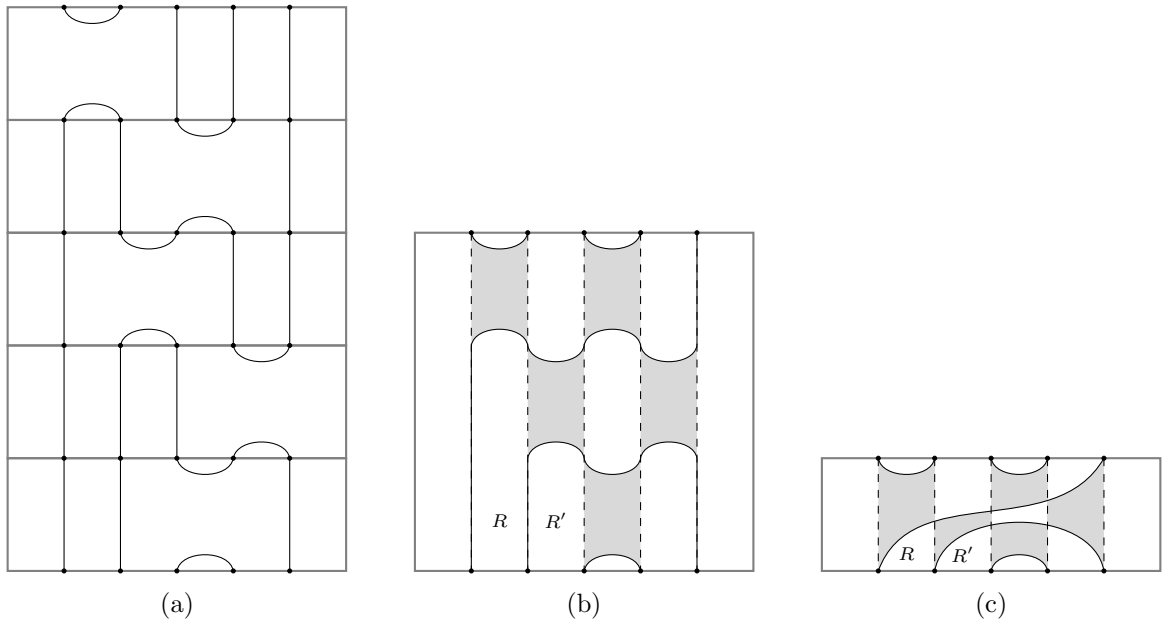
\begin{figure}[!h]
\centering
\begin{subfigure}[b]{0.32\linewidth}
\centering
\begin{tikzpicture}[scale=.75]
\draw[gray,thick] (0,0) rectangle (6,2);
\draw[gray,thick] (0,2) rectangle (6,4);
\draw[gray,thick] (0,4) rectangle (6,6);
\draw[gray,thick] (0,6) rectangle (6,8);
\draw[gray,thick] (0,8) rectangle (6,10);
\foreach \x in {1,2,3,4,5} \filldraw (\x,0) circle (1pt);
\foreach \x in {1,2,3,4,5} \filldraw (\x,2) circle (1pt);
\foreach \x in {1,2,3,4,5} \filldraw (\x,4) circle (1pt);
\foreach \x in {1,2,3,4,5} \filldraw (\x,6) circle (1pt);
\foreach \x in {1,2,3,4,5} \filldraw (\x,8) circle (1pt);
\foreach \x in {1,2,3,4,5} \filldraw (\x,10) circle (1pt);
\draw (1,0) to (1,2);
\draw (2,0) to (2,2);
\draw (3,0) to [bend left=80] (4,0);
\draw (3,2) to [bend right=80] (4,2);
\draw (5,0) to (5,2);
\draw (1,2) to (1,4);
\draw (2,2) to (2,4);
\draw (3,2) to (3,4);
\draw (4,2) to [bend left=80] (5,2);
\draw (4,4) to [bend right=80] (5,4);
\draw (1,4) to (1,6);
\draw (2,4) to [bend left=80] (3,4);
\draw (2,6) to [bend right=80] (3,6);
\draw (4,4) to (4,6);
\draw (5,4) to (5,6);
\draw (1,6) to (1,8);
\draw (2,6) to (2,8);
\draw (3,6) to [bend left=80] (4,6);
\draw (3,8) to [bend right=80] (4,8);
\draw (5,6) to (5,8);
\draw (1,8) to [bend left=80] (2,8);
\draw (1,10) to [bend right=80] (2,10);
\draw (3,8) to (3,10);
\draw (4,8) to (4,10);
\draw (5,8) to (5,10);
\end{tikzpicture}
\caption{}\label{fig:SimpleRepnA}
\end{subfigure}
\begin{subfigure}[b]{0.32\linewidth}
\centering
\begin{tikzpicture}[scale=.75]
\draw[gray,thick] (0,0) rectangle (6,6);
\fill[gray!30] (3,0) to [bend left=80] (4,0) to (4,2) to [bend left=80] (3,2);
\fill[gray!30] (2,2) to [bend left=80] (3,2) to (3,4) to [bend left=80] (2,4);
\fill[gray!30] (4,2) to [bend left=80] (5,2) to (5,4) to [bend left=80] (4,4);
\fill[gray!30] (1,4) to [bend left=80] (2,4) to (2,6) to [bend left=80] (1,6);
\fill[gray!30] (3,4) to [bend left=80] (4,4) to (4,6) to [bend left=80] (3,6);
\foreach \x in {1,2,3,4,5} \filldraw (\x,0) circle (1pt);
\foreach \x in {1,2,3,4,5} \filldraw (\x,6) circle (1pt);
\foreach \x in {1,2,3,4,5} \draw [dashed] (\x,6) to (\x,0);
\draw (3,0) to [bend left=80] (4,0);
\draw (3,2) to [bend right=80] (4,2);
\draw (4,2) to [bend left=80] (5,2);
\draw (4,4) to [bend right=80] (5,4);
\draw (2,2) to [bend left=80] (3,2);
\draw (2,4) to [bend right=80] (3,4);
\draw (3,4) to [bend left=80] (4,4);
\draw (3,6) to [bend right=80] (4,6);
\draw (1,4) to [bend left=80] (2,4);
\draw (1,6) to [bend right=80] (2,6);
\draw (5,4) to (5,6);
\draw (5,0) to (5,2);
\draw (2,0) to (2,2);
\draw (1,0) to (1,4);
\node at (1.5, .9) {\tiny $R$};
\node at (2.5, .9) {\tiny $R'$};
\end{tikzpicture}
\caption{}
\label{fig:SimpleRepnB}
\end{subfigure}
\begin{subfigure}[b]{0.32\linewidth}
\centering
\begin{tikzpicture}[scale=.75]
\draw[gray,thick] (0,0) rectangle (6,2);
\begin{scope}
\clip{(1,0) rectangle (2,2)};
\fill[gray!30] (1,0) to [out=70, in=-120] (5,2) to (2,2) to [bend left=80] (1,2);
\end{scope}
\begin{scope}
\clip{(2,0) rectangle (3,2)};
\fill[gray!30](2,0)to [bend left=80] (5,0)--(5,2) to [in=70, out=-120](1,0);
\end{scope}
\begin{scope}
\clip{(3,0) rectangle (4,2)};
\fill[gray!30](3,0) to [bend left=80] (4,0)--(5,0)to [bend right=80](2,0);
\fill[gray!30](3,2) to [bend right=80] (4,2)--(5,2) to [in=70, out=-120] (1,0);
\end{scope}
\begin{scope}
\clip{(4,0) rectangle (5,2)};
\fill[gray!30](1,0) to [out=70, in=-120] (5,2) -- (5,0) to [bend right=80] (2,0);
\end{scope}
\foreach \x in {1,2,3,4,5} \filldraw (\x,0) circle (1pt);
\foreach \x in {1,2,3,4,5} \filldraw (\x,2) circle (1pt);
\foreach \x in {1,2,3,4,5} \draw [dashed] (\x,2) to (\x,0);
\draw (1,0) to [out=70, in=-120] (5,2);
\draw (2,0) to [bend left=80] (5,0);
\draw (3,0) to [bend left=80] (4,0);
\draw (1,2) to [bend right=80] (2,2);
\draw (3,2) to [bend right=80] (4,2);
\node at (1.5, .3) {\tiny $R$};
\node at (2.5, .3) {\tiny $R'$};
\end{tikzpicture}
\caption{}
\label{fig:SimpleRepnC}
\end{subfigure}
\caption{Multiplication of simple diagrams together with the corresponding simple representation and resulting product. The $1$-regions of the simple representation and product have been shaded.}\label{fig:SimpleRepn}
\end{figure}

In light of Proposition~\ref{prop:TL=DTL}, it follows that if $d$ is a loop-free diagram from $\DTL(A_n)$, then there exists a unique $w\in\FC(A_n)$ such that $d_w=d$. The upshot is that it makes sense to refer to the simple representation of $d$.


\section{Main Results}\label{sec:MainResults}

In this section, we assume that all diagrams are loop-free and that no edge ever double-backs on itself (i.e., other than vertical propagating edges, edges never intersect any vertical line more than once). If $d$ is a $k$-diagram, we section the corresponding $k$-box into \emph{columns} by connecting node $i$ in the north face to node $i'$ in the south face. The $i$th column $C_i$ lies between nodes $i$ and $i+1$. The connected components of the complement of the edges in each column are called \emph{regions}. For example, the columns and regions for the diagram given in Figure~\ref{fig:DiagramToGraphToHeapA} are depicted in Figure~\ref{fig:DiagramToGraphToHeapB}.

\begin{figure}[!h]
\centering
\begin{subfigure}[b]{0.49\linewidth}
\centering
\begin{tikzpicture}[scale=.75]
\draw[gray, thick] (0,0) rectangle (10,3);
\foreach \x in {1,2,3,4,5,6,7,8,9} \filldraw (\x,0) circle (1pt);
\foreach \x in {1,2,3,4,5,6,7,8,9} \filldraw (\x,3) circle (1pt);
\draw (1,0) to [bend left=60] (2,0);
\draw (3,3) to [bend right=60] (6, 3);
\draw (4,3) to [bend right=60] (5,3);
\draw (6,0) to [bend left=60] (9,0);
\draw (7,0) to [bend left=60] (8,0);
\draw (1,3) to [out=-80,in=90]  (3,0);
\draw (2,3) to [out=-80,in=90]  (4,0);
\draw (9,3) to [out=-110,in=60]  (5,0);
\draw (7,3) to [bend right=60] (8,3);
\end{tikzpicture}
\caption{A loop-free diagram $d$}
\label{fig:DiagramToGraphToHeapA}
\end{subfigure}
\begin{subfigure}[b]{0.49\linewidth}
\centering
\vspace{.25cm}
\begin{tikzpicture}[scale=.75]
\draw[gray, thick] (0,0) rectangle (10,3);
\foreach \x in {1,2,3,4,5,6,7,8,9} \filldraw (\x,0) circle (1pt);
\foreach \x in {1,2,3,4,5,6,7,8,9} \filldraw (\x,3) circle (1pt);
\draw (1,0) to [bend left=60] (2,0);
\draw (3,3) to [bend right=60] (6, 3);
\draw (4,3) to [bend right=60] (5,3);
\draw (6,0) to [bend left=60] (9,0);
\draw (7,0) to [bend left=60] (8,0);
\draw (1,3) to [out=-80,in=90]  (3,0);
\draw (2,3) to [out=-80,in=90]  (4,0);
\draw (9,3) to [out=-110,in=60]  (5,0);
\draw (7,3) to [bend right=60] (8,3);
\draw [dashed] (1,0) to (1,3);
\draw [dashed] (2,0) to (2,3);
\draw [dashed] (3,0) to (3,3);
\draw [dashed] (4,0) to (4,3);
\draw [dashed] (5,0) to (5,3);
\draw [dashed] (6,0) to (6,3);
\draw [dashed] (7,0) to (7,3);
\draw [dashed] (8,0) to (8,3);
\draw [dashed] (9,0) to (9,3);
\node at (7.5, 1.15) {\tiny $R'$};
\node at (6.5, .9) {\tiny $R$};
\end{tikzpicture}
\caption{Regions $R$ and $R'$ are horizontally adjacent}
\label{fig:DiagramToGraphToHeapB}
\end{subfigure}
\par\bigskip
\begin{subfigure}[b]{0.49\linewidth}
\centering
\begin{tikzpicture}[scale=.75]
\begin{scope}
\clip{(1,0) rectangle (2,3)};
\fill[gray!30] (1,3) to [out=-80,in=90]  (3,0) to (2,0) to [bend right=60] (1,0);
\end{scope}
\begin{scope}
\clip{(2,0) rectangle (3,3)};
\fill[gray!30](2,1.57)--(3,0)--(3,1.4)--(2,3);
\fill[white] (1,3) to [out=-80,in=90]  (3,0);
\fill[white](2,3) to [out=-80,in=90]  (4,0);
\fill[gray!30](2,1.7)--(3,1.3)--(3,1.58)--(2,2.45);
\end{scope}
\begin{scope}
\clip{(3,0) rectangle (4,3)};
\fill[gray!30](3,1.57)--(4,0)--(4,3)--(3,3);
\fill[white] (3,3) to [bend right=60] (6, 3);
\fill[white](2,3) to [out=-80,in=90]  (4,0);
\end{scope}
\begin{scope}
\clip{(4,0) rectangle (5,3)};
\fill[gray!30](3,3) to [bend right=60] (6, 3);
\fill[white](4,3) to [bend right=60] (5,3);
\end{scope}
\begin{scope}
\clip{(5,0) rectangle (6,3)};
\fill[gray!30](5,0)--(6,.9)--(6,3)--(5,3);
\fill[white] (3,3) to [bend right=60] (6, 3);
\fill[white](9,3) to [out=-110,in=60]  (5,0);
\end{scope}
\begin{scope}
\clip{(6,0) rectangle (7,3)};
\fill[gray!30](9,3) to [out=-110,in=60]  (5,0) to (6,0) to [bend left=60] (9,0);
\end{scope}
\begin{scope}
\clip{(7,0) rectangle (8,3)};
\fill[gray!30](7,1.5)--(8,1.9) -- (8,3) -- (7,3);
\fill[gray!30](9,3) to [out=-110,in=60]  (5,0);
\fill[white] (7,3) to [bend right=60] (8,3);
\fill[gray!30](6,0) to [bend left=60] (9,0);
\fill[white] (7,0) to [bend left=60] (8,0);
\end{scope}
\begin{scope}
\clip{(8,0) rectangle (9,3)};
\fill[gray!30](8,0)--(9,0)--(9,3)--(8,1.9);
\fill[white] (6,0) to [bend left=60] (9,0);
\fill[white](9,3) to [out=-110,in=60]  (5,0);
\end{scope}
\draw[gray, thick] (0,0) rectangle (10,3);
\foreach \x in {1,2,3,4,5,6,7,8,9} \filldraw (\x,0) circle (1pt);
\foreach \x in {1,2,3,4,5,6,7,8,9} \filldraw (\x,3) circle (1pt);
\draw (1,0) to [bend left=60] (2,0);
\draw (3,3) to [bend right=60] (6, 3);
\draw (4,3) to [bend right=60] (5,3);
\draw (6,0) to [bend left=60] (9,0);
\draw (7,0) to [bend left=60] (8,0);
\draw (1,3) to [out=-80,in=90]  (3,0);
\draw (2,3) to [out=-80,in=90]  (4,0);
\draw (9,3) to [out=-110,in=60]  (5,0);
\draw (7,3) to [bend right=60] (8,3);
\draw [dashed] (1,0) to (1,3);
\draw [dashed] (2,0) to (2,3);
\draw [dashed] (3,0) to (3,3);
\draw [dashed] (4,0) to (4,3);
\draw [dashed] (5,0) to (5,3);
\draw [dashed] (6,0) to (6,3);
\draw [dashed] (7,0) to (7,3);
\draw [dashed] (8,0) to (8,3);
\draw [dashed] (9,0) to (9,3);
\node at (4.5, 2.5) {\tiny $A$};
\node at (7.5, 2.25) {\tiny $D$};
\node at (7.5, 0.5) {\tiny $I$};
\node at (8.5, 1.15) {\tiny $G$};
\node at (6.5, .9) {\tiny $F$};
\node at (5.5, 1.5) {\tiny $C$};
\node at (3.5, 1.85) {\tiny $B$};
\node at (2.5, 1.55) {\tiny $E$};
\node at (1.5, 1.15) {\tiny $H$};
\end{tikzpicture}
\caption{Shaded $1$-regions of $d$}
\label{fig:DiagramToGraphToHeapC}
\end{subfigure}
\begin{subfigure}[b]{0.49\linewidth}
\centering
\begin{tikzpicture}[scale=.75]
\begin{scope}
\clip{(1,0) rectangle (2,3)};
\fill[gray!30] (1,3) to [out=-80,in=90]  (3,0) to (2,0) to [bend right=60] (1,0);
\end{scope}
\begin{scope}
\clip{(2,0) rectangle (3,3)};
\fill[gray!30](2,1.57)--(3,0)--(3,1.4)--(2,3);
\fill[white] (1,3) to [out=-80,in=90]  (3,0);
\fill[white](2,3) to [out=-80,in=90]  (4,0);
\fill[gray!30](2,1.7)--(3,1.3)--(3,1.58)--(2,2.45);
\end{scope}
\begin{scope}
\clip{(3,0) rectangle (4,3)};
\fill[gray!30](3,1.57)--(4,0)--(4,3)--(3,3);
\fill[white] (3,3) to [bend right=60] (6, 3);
\fill[white](2,3) to [out=-80,in=90]  (4,0);
\end{scope}
\begin{scope}
\clip{(4,0) rectangle (5,3)};
\fill[gray!30](3,3) to [bend right=60] (6, 3);
\fill[white](4,3) to [bend right=60] (5,3);
\end{scope}
\begin{scope}
\clip{(5,0) rectangle (6,3)};
\fill[gray!30](5,0)--(6,.9)--(6,3)--(5,3);
\fill[white] (3,3) to [bend right=60] (6, 3);
\fill[white](9,3) to [out=-110,in=60]  (5,0);
\end{scope}
\begin{scope}
\clip{(6,0) rectangle (7,3)};
\fill[gray!30](9,3) to [out=-110,in=60]  (5,0) to (6,0) to [bend left=60] (9,0);
\end{scope}
\begin{scope}
\clip{(7,0) rectangle (8,3)};
\fill[gray!30](7,1.5)--(8,1.9) -- (8,3) -- (7,3);
\fill[gray!30](9,3) to [out=-110,in=60]  (5,0);
\fill[white] (7,3) to [bend right=60] (8,3);
\fill[gray!30](6,0) to [bend left=60] (9,0);
\fill[white] (7,0) to [bend left=60] (8,0);
\end{scope}
\begin{scope}
\clip{(8,0) rectangle (9,3)};
\fill[gray!30](8,0)--(9,0)--(9,3)--(8,1.9);
\fill[white] (6,0) to [bend left=60] (9,0);
\fill[white](9,3) to [out=-110,in=60]  (5,0);
\end{scope}
\draw[gray, thick] (0,0) rectangle (10,3);
\foreach \x in {1,2,3,4,5,6,7,8,9} \filldraw (\x,0) circle (1pt);
\foreach \x in {1,2,3,4,5,6,7,8,9} \filldraw (\x,3) circle (1pt);
\draw (1,0) to [bend left=60] (2,0);
\draw (3,3) to [bend right=60] (6, 3);
\draw (4,3) to [bend right=60] (5,3);
\draw (6,0) to [bend left=60] (9,0);
\draw (7,0) to [bend left=60] (8,0);
\draw (1,3) to [out=-80,in=90]  (3,0);
\draw (2,3) to [out=-80,in=90]  (4,0);
\draw (9,3) to [out=-110,in=60]  (5,0);
\draw (7,3) to [bend right=60] (8,3);
\draw [dashed] (1,0) to (1,3);
\draw [dashed] (2,0) to (2,3);
\draw [dashed] (3,0) to (3,3);
\draw [dashed] (4,0) to (4,3);
\draw [dashed] (5,0) to (5,3);
\draw [dashed] (6,0) to (6,3);
\draw [dashed] (7,0) to (7,3);
\draw [dashed] (8,0) to (8,3);
\draw [dashed] (9,0) to (9,3);
\node[circle, fill=black, inner sep=1.5pt](4-1) at (4.5, 2.5) {};
\node[circle, fill=black, inner sep=1.5pt](7-1) at (7.5, 2.25) {};
\node[circle, fill=black, inner sep=1.5pt](7-2) at (7.5, 0.5) {};
\node[circle, fill=black, inner sep=1.5pt](8-1) at (8.5, 1.15) {};
\node[circle, fill=black, inner sep=1.5pt](6-1) at (6.5, 0.9) {};
\node[circle, fill=black, inner sep=1.5pt](5-1) at (5.5, 1.5) {};
\node[circle, fill=black, inner sep=1.5pt](3-1) at (3.5, 1.85) {};
\node[circle, fill=black, inner sep=1.5pt](2-1) at (2.5, 1.55) {};
\node[circle, fill=black, inner sep=1.5pt](1-1) at (1.5, 1.15) {};
\draw[-stealth] (4-1) to (3-1);
\draw[-stealth] (3-1) to (2-1);
\draw[-stealth] (2-1) to (1-1);
\draw[-stealth] (4-1) to (5-1);
\draw[-stealth] (5-1) to (6-1);
\draw[-stealth] (6-1) to (7-2);
\draw[-stealth] (7-1) to (6-1);
\draw[-stealth] (7-1) to (8-1);
\draw[-stealth] (8-1) to (7-2);
\end{tikzpicture}
\caption{Directed graph $G_d$}
\label{fig:DiagramToGraphToHeapD}
\end{subfigure}
\par\bigskip
\begin{subfigure}[b]{0.49\linewidth}
\centering
\begin{tikzpicture}[scale=.75]
\begin{scope}
\clip{(1,0) rectangle (2,3)};
\fill[gray!30] (1,3) to [out=-80,in=90]  (3,0) to (2,0) to [bend right=60] (1,0);
\end{scope}
\begin{scope}
\clip{(2,0) rectangle (3,3)};
\fill[gray!30](2,1.57)--(3,0)--(3,1.4)--(2,3);
\fill[white] (1,3) to [out=-80,in=90]  (3,0);
\fill[white](2,3) to [out=-80,in=90]  (4,0);
\fill[gray!30](2,1.7)--(3,1.3)--(3,1.58)--(2,2.45);
\end{scope}
\begin{scope}
\clip{(3,0) rectangle (4,3)};
\fill[gray!30](3,1.57)--(4,0)--(4,3)--(3,3);
\fill[white] (3,3) to [bend right=60] (6, 3);
\fill[white](2,3) to [out=-80,in=90]  (4,0);
\end{scope}
\begin{scope}
\clip{(4,0) rectangle (5,3)};
\fill[gray!30](3,3) to [bend right=60] (6, 3);
\fill[white](4,3) to [bend right=60] (5,3);
\end{scope}
\begin{scope}
\clip{(5,0) rectangle (6,3)};
\fill[gray!30](5,0)--(6,.9)--(6,3)--(5,3);
\fill[white] (3,3) to [bend right=60] (6, 3);
\fill[white](9,3) to [out=-110,in=60]  (5,0);
\end{scope}
\begin{scope}
\clip{(6,0) rectangle (7,3)};
\fill[gray!30](9,3) to [out=-110,in=60]  (5,0) to (6,0) to [bend left=60] (9,0);
\end{scope}
\begin{scope}
\clip{(7,0) rectangle (8,3)};
\fill[gray!30](7,1.5)--(8,1.9) -- (8,3) -- (7,3);
\fill[gray!30](9,3) to [out=-110,in=60]  (5,0);
\fill[white] (7,3) to [bend right=60] (8,3);
\fill[gray!30](6,0) to [bend left=60] (9,0);
\fill[white] (7,0) to [bend left=60] (8,0);
\end{scope}
\begin{scope}
\clip{(8,0) rectangle (9,3)};
\fill[gray!30](8,0)--(9,0)--(9,3)--(8,1.9);
\fill[white] (6,0) to [bend left=60] (9,0);
\fill[white](9,3) to [out=-110,in=60]  (5,0);
\end{scope}
\draw[gray, thick] (0,0) rectangle (10,3);
\foreach \x in {1,2,3,4,5,6,7,8,9} \filldraw (\x,0) circle (1pt);
\foreach \x in {1,2,3,4,5,6,7,8,9} \filldraw (\x,3) circle (1pt);
\draw (1,0) to [bend left=60] (2,0);
\draw (3,3) to [bend right=60] (6, 3);
\draw (4,3) to [bend right=60] (5,3);
\draw (6,0) to [bend left=60] (9,0);
\draw (7,0) to [bend left=60] (8,0);
\draw (1,3) to [out=-80,in=90]  (3,0);
\draw (2,3) to [out=-80,in=90]  (4,0);
\draw (9,3) to [out=-110,in=60]  (5,0);
\draw (7,3) to [bend right=60] (8,3);
\draw [dashed] (1,0) to (1,3);
\draw [dashed] (2,0) to (2,3);
\draw [dashed] (3,0) to (3,3);
\draw [dashed] (4,0) to (4,3);
\draw [dashed] (5,0) to (5,3);
\draw [dashed] (6,0) to (6,3);
\draw [dashed] (7,0) to (7,3);
\draw [dashed] (8,0) to (8,3);
\draw [dashed] (9,0) to (9,3);
\node[circle, fill=black, inner sep=1.5pt, label={[xshift=0cm, yshift=-.6cm]\tiny $s_4$}](4-1) at (4.5, 2.5) {};
\node[circle, fill=black, inner sep=1.5pt, label={[xshift=0cm, yshift=-.05cm]\tiny $s_7$}](7-1) at (7.5, 2.1) {};
\node[circle, fill=black, inner sep=1.5pt, label={[xshift=0cm, yshift=0cm]\tiny $s_7$}](7-2) at (7.5, 0.5) {};
\node[circle, fill=black, inner sep=1.5pt, label={[xshift=0cm, yshift=0cm]\tiny $s_8$}](8-1) at (8.5, 1.15) {};
\node[circle, fill=black, inner sep=1.5pt, label={[xshift=0cm, yshift=.1cm]\tiny $s_6$}](6-1) at (6.5, 0.9) {};
\node[circle, fill=black, inner sep=1.5pt, label={[xshift=0cm, yshift=-.55cm]\tiny $s_5$}](5-1) at (5.5, 1.5) {};
\node[circle, fill=black, inner sep=1.5pt, label={[xshift=0cm, yshift=-.5cm]\tiny $s_3$}](3-1) at (3.5, 1.85) {};
\node[circle, fill=black, inner sep=1.5pt, label={[xshift=.2cm, yshift=-.5cm]\tiny $s_2$}](2-1) at (2.5, 1.55) {};
\node[circle, fill=black, inner sep=1.5pt, label={[xshift=0cm, yshift=-.55cm]\tiny $s_1$}](1-1) at (1.5, 1.15) {};
\draw[-stealth] (4-1) to (3-1);
\draw[-stealth] (3-1) to (2-1);
\draw[-stealth] (2-1) to (1-1);
\draw[-stealth] (4-1) to (5-1);
\draw[-stealth] (5-1) to (6-1);
\draw[-stealth] (6-1) to (7-2);
\draw[-stealth] (7-1) to (6-1);
\draw[-stealth] (7-1) to (8-1);
\draw[-stealth] (8-1) to (7-2);
\end{tikzpicture}
\caption{Labeled directed graph $G_d^S$}
\label{fig:DiagramToGraphToHeapE}
\end{subfigure}
\begin{subfigure}[b]{0.49\linewidth}
\centering
\begin{tikzpicture}[scale=.55]
\heapblock{1}{1}{s_1}{black};
\heapblock{7}{1}{s_7}{black};
\heapblock{2}{2}{s_2}{black};
\heapblock{6}{2}{s_6}{black};
\heapblock{8}{3}{s_8}{black};
\heapblock{3}{3}{s_3}{black};
\heapblock{5}{3}{s_5}{black};
\heapblock{7}{4}{s_7}{black};
\heapblock{4}{4}{s_4}{black};
\end{tikzpicture}
\caption{Heap corresponding to $G_d^S$}
\label{fig:DiagramToGraphToHeapF}
\end{subfigure}
\caption{Shaded $1$-regions, directed graph $G_d$, and labeled directed graph $G_d^S$ for a diagram $d$ together with the lattice point representation of the corresponding heap.}
\label{fig:DiagramToGraphToHeap}
\end{figure}
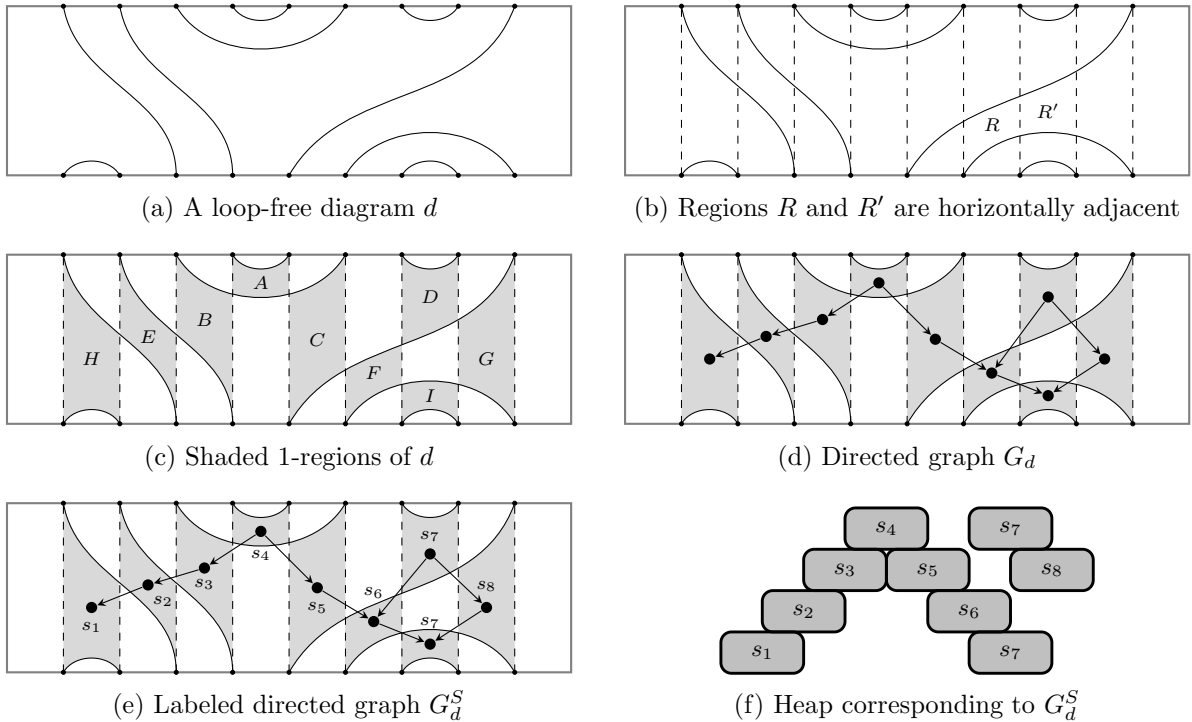

\begin{lemma}\label{lem:EvenEdges}
The number of edges within a single column of a diagram is even.
\end{lemma}

\begin{proof}
This is clear for the simple representation of a diagram as each simple diagram $d_i$ contributes precisely two edges to the column $C_i$. Isotopically deforming edges (while avoiding edges doubling-back on themselves) does not change the parity of the number of edges within the column.
\end{proof}

\begin{lemma}\label{lem:OddRegions}
The number of regions in each column is odd.
\end{lemma}

\begin{proof}
Since the number of edges within any column is even, there must be an odd number of regions within each column.
\end{proof}

Note that isotopically deforming the edges of a concrete diagram preserves the relative adjacency of the regions in each column. We say that regions $R$ and $R'$ of column $C$ are \emph{vertically adjacent} if they are adjacent across a common edge. Within a single column, we will label the first region just below the north face with a $0$. Moving south, the next region will be labeled with a $1$ and we continue this way, alternating labels $0$ and $1$. We will refer to the labeled regions as $0$-regions and $1$-regions, respectively. By Lemma~\ref{lem:OddRegions}, it is clear that the southernmost region in each column is a $0$-region. Figure~\ref{fig:DiagramToGraphToHeapC} depicts a diagram and its $0$-regions and $1$-regions where we have shaded the $1$-regions.

Observe that if $d$ is a diagram from $\DTL(A_n)$, then each $1$-region in column $C_i$ of the simple representation for $d$ corresponds precisely to the regions bounded above and below by the pair of edges corresponding to a unique factor $d_i$.

Suppose $R$ and $R'$ are regions of adjacent columns $C$ and $C'$, respectively, of some diagram $d$.  We say that $R$ and $R'$ are \emph{horizontally adjacent} if there exist points $p$ and $p'$ in $R$ and $R'$, respectively, such that the line segment joining $p$ and $p'$ does not cross any edge of $d$.  Loosely speaking, $R$ and $R'$ are horizontally adjacent if they are adjacent across the common vertical boundary of $C$ and $C'$.  

Since we forbid edges from doubling-back on themselves, horizontal adjacency of regions is preserved when isotopically deforming the edges of $d$.  This implies that horizontal adjacency is well-defined.

Figure~\ref{fig:DiagramToGraphToHeapB} depicts two horizontally adjacent regions, $R$ and $R'$.  However, the regions labeled $R$ and $R'$ in the simple representation depicted in Figure~\ref{fig:SimpleRepnB} may appear at first glance to be horizontally adjacent, but they are not. This is evident by looking at the corresponding regions $R$ and $R'$ in Figure~\ref{fig:SimpleRepnC}.

If $R$ is a region of column $C$ in diagram $d$, then the \emph{depth} of $R$, $\depth(R)$, is defined to be the number of regions in $C$ strictly between the north face of $d$ and $R$. For example, we have $\depth(R)=1$ and $\depth(R')=2$ for the regions $R$ and $R'$ in Figure~\ref{fig:DiagramToGraphToHeapB}. Note that for any diagram, the northernmost region in each column has depth 0.  Moreover, every $1$-region has an odd depth while every $0$-region has an even depth.

\begin{lemma}\label{lem:HorizontallyAdjacentRegions}
If $R$ and $R'$ are horizontally adjacent regions of a diagram $d$, then 
\begin{enumerate}
\item $\abs{\depth(R)-\depth(R')}=1$, and
\item $R$ is a $1$-region if and only if $R'$ is a $0$-region.
\end{enumerate}
\end{lemma}

\begin{proof}
Induction on depth quickly yields (1) while (2) is an immediate consequence of (1).
\end{proof}

We say that regions $R$ and $R'$ of a diagram $d$ are \emph{diagonally adjacent} if there exists a region $S$ that is vertically adjacent to $R$ and horizontally adjacent to $R'$.  In particular, if $S$ lies below $R$, then we write $R \to R'$.

\begin{lemma}\label{lem:DiagLabel}
If $R\to R'$ in a diagram $d$, then $R$ is a $1$-region if and only if $R'$ is a $1$-region.
\end{lemma}

\begin{proof}
The result follows immediately from the construction of $0$-regions and $1$-regions together with Lemma~\ref{lem:HorizontallyAdjacentRegions}.
\end{proof}

\begin{example}
In Figure~\ref{fig:DiagramToGraphToHeapC}, we see that $A \to B \to E \to H$, $A \to C \to F \to I$, $D \to F \to I$ and $D \to G \to I$.
\end{example}

\begin{remark}\label{rem:Checkerboard}
If $C_i$ is not the leftmost or rightmost column, edges bounding $1$-regions in column $C_i$ must pass into its adjacent columns unless an edge connects directly to a node at the top or bottom of $C_i$. In the leftmost column, $C_1$, no edge will pass through to the left. Similarly, in the rightmost column, $C_n$, no edge will pass through to the right. This implies that if $R$ and $R'$ are both $1$-regions in the same column $C_i$ with $\depth(R')=\depth(R)+2$ (i.e., $R$ and $R'$ are consecutive $1$-regions in $C_i$ with $R'$ below $R$), then there exist $1$-regions $T$ and $T'$ in $C_{i-1}$ and $C_{i+1}$, respectively, such that $R \to T \to R'$ and $R \to T' \to R'$.  Loosely speaking, this determines a local checkerboard pattern of $1$-regions as seen in Figure~\ref{fig:DiagramToGraphToHeapC}.
\end{remark}

The checkerboard pattern of $0$-regions and $1$-regions motivates the following definition. Let $d$ be a diagram having $1$-regions $R_1, \ldots, R_n$. Define $G_d$ to be the directed graph having
\begin{enumerate}
\item vertex set $V(G_d):=\{R_1, \ldots R_n\}$ and
\item directed edges $(R_k, R_l)$ whenever $R_k \to R_l$.
\end{enumerate}

Since we require the edges of $d$ to not double-back on themselves, it is clear that $G_d$ is independent of choice of concrete representation for $d$; indeed, isotopically deforming the edges and rescaling the rectangle preserves horizontal and vertical adjacency and so diagonal adjacency is also preserved. In particular, if $w$ indexes $d$, then we can construct $G_d$ using the simple representation of $d_w$.

Figure~\ref{fig:DiagramToGraphToHeapD} shows the directed graph $G_d$ for the diagram $d$ given in Figure~\ref{fig:DiagramToGraphToHeapA}. Observe that directed paths correspond to chains of diagonally adjacent regions.

Next, we will append labels from the generating set of the Coxeter group to the vertices of $G_d$. Define the vertex labeling function $\nu : V(G_d) \to S$ as follows. If $R$ is a $1$-region that lies in column $i$, then $\nu(R) = s_i$. That is, each region is labeled with the generator of the corresponding column. Now, define $G_d^S$ to be the directed graph $G_d$ together with the labels on the vertices assigned by $\nu$. Figure~\ref{fig:DiagramToGraphToHeapE} shows the labeled directed graph $G_d^S$ for the diagram $d$ given in Figure~\ref{fig:DiagramToGraphToHeapA}.

Each labeled directed graph $G_d^S$ naturally corresponds to a unique labeled Hasse diagram of a heap for some element in $W(A_n)$. It follows from Remark~\ref{rem:Checkerboard} and Proposition~\ref{prop:Stembridge} that this element is $\FC$. Figure~\ref{fig:DiagramToGraphToHeapF} shows the heap that corresponds to the diagram $d$ given in Figure~\ref{fig:DiagramToGraphToHeapA}. It remains to show that the heap determined by $G_d^S$ corresponds to the group element that indexes the diagram $d$.

Since diagonal adjacency is preserved when isotopically deforming the edges of $d$, as in the simple representation, a $1$-region $R$ in column $C_i$ is bounded above and below by a pair of edges corresponding to the simple diagram $d_i$. This region is labeled $s_i$ in $G_d^S$. The structure of $G_d^S$ determines $w\in\FC(A_n)$ satisfying $d = d_w$, and it follows that $G_d^S$ corresponds to $H(w)$. Note that we find $w$ by writing the elements of $S$ corresponding to the labels of each row of the heap left to right starting from the top row and working toward the bottom of the heap.

The above discussion together with the preceding lemmas justifies the following theorem.

\begin{theorem}\label{thm:MainResult}
If $d$ is a loop-free diagram in $\TL(A_n)$, then $d$ is indexed by the heap determined by $G_d^S$. \qed
\end{theorem}

An immediate consequence of the above theorem is that we non-recursively obtain a factorization of a diagram $d$ by reading off from top to bottom and left to right the entries of the heap determined by $G_d^S$.  If we choose the canonical representation of the heap, then the factorization of $d$ corresponds to the Cartier--Foata normal form of~\cite{Cartier1969,Green2006a}. Our construction also yields the following corollary, which appeared independently as Lemma~3.3 in~\cite{Green1998a}.

\begin{corollary}\label{cor:MainResult}
If $d$ is a loop-free diagram in $\TL(A_n)$, then the number of occurrences of the simple diagram $d_i$ in any factorization for $d$ is equal to half the number of edges passing through the column $C_i$.
\end{corollary}

\begin{example}\label{ex:MainResult}
Consider the diagram $d$ given in Figure~\ref{fig:MainResultA}.  After forming columns, we obtain a checkerboard of $0$-regions and $1$-regions, which yields the labeled directed graph $G_d^S$ depicted in Figure~\ref{fig:MainResultB}. Then $G_d^S$ determines the canonical representation of the heap given in Figure~\ref{fig:MainResultC}, where each row of the heap has a unique color.  In Figure~\ref{fig:MainResultD}, we have color-coded the $1$-regions of $d$ to match the corresponding entries in the heap.  By reading off the entries of the heap, we see that $d = d_w$ where
\[
w = {\color{butterfly}s_{2}s_6} {\color{springgreen}s_{1}s_{3}} {\color{rose}s_{2}s_{4}} {\color{darkorchid}s_{3}s_{5}} {\color{nectarine}s_{4}}.
\]
Equivalently, we obtain the factorization
\[
d={\color{butterfly}d_{2}d_6} {\color{springgreen}d_{1}d_{3}} {\color{rose}d_{2}d_{4}} {\color{darkorchid}d_{3}d_{5}} {\color{nectarine}d_{4}},
\]
which is shown (rotated counterclockwise by a quarter turn in the interest of space) in Figure~\ref{fig:MainResultE}.
\end{example}

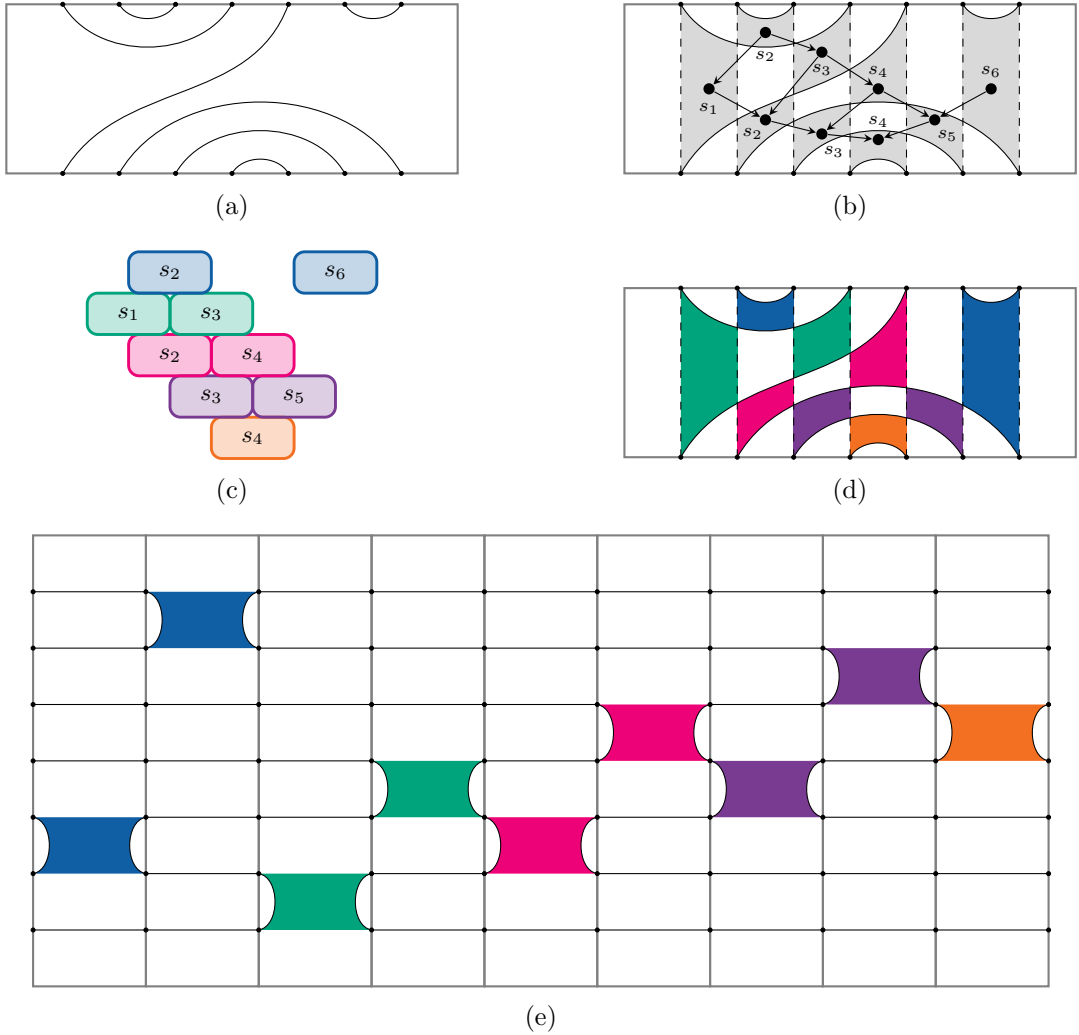
\begin{figure}
\centering
\begin{subfigure}[b]{0.49\linewidth}
\centering
\begin{tikzpicture}[scale=.75]
\draw[gray, thick] (0,0) rectangle (8,3);
\foreach \x in {1,2,3,4,5,6,7} \filldraw (\x,0) circle (1pt);
\foreach \x in {1,2,3,4,5,6,7} \filldraw (\x,3) circle (1pt);
\draw (2,0) to [bend left=60] (7,0);
\draw (1,3) to [bend right=60] (4, 3);
\draw (2,3) to [bend right=60] (3,3);
\draw (3,0) to [bend left=60] (6,0);
\draw (4,0) to [bend left=60] (5,0);
\draw (5,3) to [out=-110,in=60]  (1,0);
\draw (6,3) to [bend right=60] (7,3);
\end{tikzpicture}
\caption{}
\label{fig:MainResultA}
\end{subfigure}
\begin{subfigure}[b]{0.49\linewidth}\centering
\begin{tikzpicture}[scale=.75]
\begin{scope}
\clip{(1,0) rectangle (2,3)};
\fill[gray!30](1,3) to [bend right=60] (4, 3) to (5,3) to [out=-110,in=60]  (1,0);
\end{scope}
\begin{scope}
\clip{(2,0) rectangle (3,3)};
\fill[gray!30](2,3) to [bend right=60] (3,3) to (4,3) to [bend left=60] (1, 3);
\fill[gray!30](5,3) to [out=-110,in=60]  (1,0) to (2,0) to [bend left=60] (7,0);
\end{scope}
\begin{scope}
\clip{(3,0) rectangle (4,3)};
\fill[gray!30](5,3) to [out=-110,in=60]  (1,0) to (1,3) to [bend right=60] (4, 3);
\fill[gray!30](2,0) to [bend left=60] (7,0) to (6,0) to [bend right=60] (3,0);
\end{scope}
\begin{scope}
\clip{(4,0) rectangle (5,3)};
\fill[gray!30](5,3) to [out=-110,in=60]  (1,0) to (2,0) to [bend left=60] (7,0);
\fill[gray!30](6,0) to [bend right=60] (3,0) to (4,0) to [bend left=60] (5,0);
\end{scope}
\begin{scope}
\clip{(5,0) rectangle (6,3)};
\fill[gray!30](2,0) to [bend left=60] (7,0) to (6,0) to [bend right=60] (3,0);
\end{scope}
\begin{scope}
\clip{(6,0) rectangle (7,3)};
\fill[gray!30](2,0) to [bend left=60] (7,0) to (7,3) to [bend left=60] (6,3);
\end{scope}
\draw[gray, thick] (0,0) rectangle (8,3);
\foreach \x in {1,2,3,4,5,6,7} \filldraw (\x,0) circle (1pt);
\foreach \x in {1,2,3,4,5,6,7} \filldraw (\x,3) circle (1pt);
\draw (2,0) to [bend left=60] (7,0);
\draw (1,3) to [bend right=60] (4, 3);
\draw (2,3) to [bend right=60] (3,3);
\draw (3,0) to [bend left=60] (6,0);
\draw (4,0) to [bend left=60] (5,0);
\draw (5,3) to [out=-110,in=60]  (1,0);
\draw (6,3) to [bend right=60] (7,3);
\draw [dashed] (1,0) to (1,3);
\draw [dashed] (2,0) to (2,3);
\draw [dashed] (3,0) to (3,3);
\draw [dashed] (4,0) to (4,3);
\draw [dashed] (5,0) to (5,3);
\draw [dashed] (6,0) to (6,3);
\draw [dashed] (7,0) to (7,3);
\node[circle, fill=black, inner sep=1.5pt, label={[xshift=0cm, yshift=-.6cm]\tiny $s_2$}](2-1) at (2.5, 2.5) {};
\node[circle, fill=black, inner sep=1.5pt, label={[xshift=0cm, yshift=-.55cm]\tiny $s_1$}](1-1) at (1.5, 1.5) {};
\node[circle, fill=black, inner sep=1.5pt, label={[xshift=-.17cm, yshift=-.47cm]\tiny $s_2$}](2-2) at (2.5, 0.95) {};
\node[circle, fill=black, inner sep=1.5pt, label={[xshift=0cm, yshift=-.52cm]\tiny $s_3$}](3-1) at (3.5, 2.15) {};
\node[circle, fill=black, inner sep=1.5pt, label={[xshift=.16cm, yshift=-.52cm]\tiny $s_3$}](3-2) at (3.5, 0.70) {};
\node[circle, fill=black, inner sep=1.5pt, label={[xshift=0cm, yshift=-.06cm]\tiny $s_4$}](4-1) at (4.5, 1.5) {};
\node[circle, fill=black, inner sep=1.5pt, label={[xshift=0cm, yshift=-.05cm]\tiny $s_4$}](4-2) at (4.5, .6) {};
\node[circle, fill=black, inner sep=1.5pt, label={[xshift=.17cm, yshift=-.5cm]\tiny $s_5$}](5-1) at (5.5, .95) {};
\node[circle, fill=black, inner sep=1.5pt, label={[xshift=0cm, yshift=-.06cm]\tiny $s_6$}](6-1) at (6.5, 1.5) {};
\draw[-stealth] (2-1) to (1-1);
\draw[-stealth] (2-1) to (3-1);
\draw[-stealth] (1-1) to (2-2);
\draw[-stealth] (3-1) to (2-2);
\draw[-stealth] (3-1) to (4-1);
\draw[-stealth] (4-1) to (3-2);
\draw[-stealth] (2-2) to (3-2);
\draw[-stealth] (4-1) to (5-1);
\draw[-stealth] (3-2) to (4-2);
\draw[-stealth] (5-1) to (4-2);
\draw[-stealth] (6-1) to (5-1);
\end{tikzpicture}
\caption{}
\label{fig:MainResultB}
\end{subfigure}
\par\bigskip
\begin{subfigure}[b]{0.49\linewidth}
\centering
\begin{tikzpicture}[scale=.55]
\heapblock{4}{1}{s_4}{nectarine};
\heapblock{3}{2}{s_3}{darkorchid};
\heapblock{5}{2}{s_5}{darkorchid};
\heapblock{2}{3}{s_2}{rose};
\heapblock{4}{3}{s_4}{rose};
\heapblock{6}{5}{s_6}{butterfly};
\heapblock{1}{4}{s_1}{springgreen};
\heapblock{3}{4}{s_3}{springgreen};
\heapblock{2}{5}{s_2}{butterfly};
\end{tikzpicture}
\caption{}
\label{fig:MainResultC}
\end{subfigure}
\begin{subfigure}[b]{0.49\linewidth}
\centering
\begin{tikzpicture}[scale=.75]
\begin{scope}
\clip{(1,0) rectangle (2,3)};
\fill[springgreen](1,3) to [bend right=60] (4, 3) to (5,3) to [out=-110,in=60]  (1,0);
\end{scope}
\begin{scope}
\clip{(2,0) rectangle (3,3)};
\fill[butterfly](2,3) to [bend right=60] (3,3) to (4,3) to [bend left=60] (1, 3);
\fill[rose](5,3) to [out=-110,in=60]  (1,0) to (2,0) to [bend left=60] (7,0);
\end{scope}
\begin{scope}
\clip{(3,0) rectangle (4,3)};
\fill[springgreen](5,3) to [out=-110,in=60]  (1,0) to (1,3) to [bend right=60] (4, 3);
\fill[darkorchid](2,0) to [bend left=60] (7,0) to (6,0) to [bend right=60] (3,0);
\end{scope}
\begin{scope}
\clip{(4,0) rectangle (5,3)};
\fill[rose](5,3) to [out=-110,in=60]  (1,0) to (2,0) to [bend left=60] (7,0);
\fill[nectarine](6,0) to [bend right=60] (3,0) to (4,0) to [bend left=60] (5,0);
\end{scope}
\begin{scope}
\clip{(5,0) rectangle (6,3)};
\fill[darkorchid](2,0) to [bend left=60] (7,0) to (6,0) to [bend right=60] (3,0);
\end{scope}
\begin{scope}
\clip{(6,0) rectangle (7,3)};
\fill[butterfly](2,0) to [bend left=60] (7,0) to (7,3) to [bend left=60] (6,3);
\end{scope}
\draw[gray, thick] (0,0) rectangle (8,3);
\foreach \x in {1,2,3,4,5,6,7} \filldraw (\x,0) circle (1pt);
\foreach \x in {1,2,3,4,5,6,7} \filldraw (\x,3) circle (1pt);
\draw (2,0) to [bend left=60] (7,0);
\draw (1,3) to [bend right=60] (4, 3);
\draw (2,3) to [bend right=60] (3,3);
\draw (3,0) to [bend left=60] (6,0);
\draw (4,0) to [bend left=60] (5,0);
\draw (5,3) to [out=-110,in=60]  (1,0);
\draw (6,3) to [bend right=60] (7,3);
\draw [dashed] (1,0) to (1,3);
\draw [dashed] (2,0) to (2,3);
\draw [dashed] (3,0) to (3,3);
\draw [dashed] (4,0) to (4,3);
\draw [dashed] (5,0) to (5,3);
\draw [dashed] (6,0) to (6,3);
\draw [dashed] (7,0) to (7,3);
\end{tikzpicture}
\caption{}
\label{fig:MainResultD}
\end{subfigure}
\par\bigskip
\begin{subfigure}[b]{1\linewidth}
\centering
\rotatebox{90}{%
\begin{tikzpicture}[scale=.75]
\begin{scope}
\clip{(1,0) rectangle (2,18)};
\fill[springgreen] (1,12) to [bend left=80] (2,12) to (2,14) to [bend left=80] (1,14);
\end{scope}
\begin{scope}
\clip{(2,0) rectangle (3,18)};
\fill[butterfly](2,16) to [bend left=80] (3,16) to (3,18) to [bend left=80] (2,18);
\fill[rose] (2,8) to [bend left=80] (3,8) to (3,10) to [bend left=80] (2,10);
\end{scope}
\begin{scope}
\clip{(3,0) rectangle (4,18)};
\fill[springgreen] (3,10) to [bend left=80] (4,10) to (4,12) to [bend left=80] (3,12);
\fill[darkorchid] (3,4) to [bend left=80] (4,4) to (4,6) to [bend left=80] (3,6);
\end{scope}
\begin{scope}
\clip{(4,0) rectangle (5,18)};
\fill[rose] (4,6) to [bend left=80] (5,6) to (5,8) to [bend left=80] (4,8);
\fill[nectarine] (4,0) to [bend left=80] (5,0) to (5,2) to [bend left=80] (4,2);
\end{scope}
\begin{scope}
\clip{(5,0) rectangle (6,18)};
\fill[darkorchid] (5,2) to [bend left=80] (6,2) to (6,4) to [bend left=80] (5,4);
\end{scope}
\begin{scope}
\clip{(6,0) rectangle (7,18)};
\fill[butterfly](6,14) to [bend left=80] (7,14)to (7,16) to [bend left=80] (6,16);
\end{scope}
\draw[gray,thick] (0,0) rectangle (8,2);
\draw[gray,thick] (0,2) rectangle (8,4);
\draw[gray,thick] (0,4) rectangle (8,6);
\draw[gray,thick] (0,6) rectangle (8,8);
\draw[gray,thick] (0,8) rectangle (8,10);
\draw[gray,thick] (0,10) rectangle (8,12);
\draw[gray,thick] (0,12) rectangle (8,14);
\draw[gray,thick] (0,14) rectangle (8,16);
\draw[gray,thick] (0,16) rectangle (8,18);
\foreach \x in {1,2,3,4,5,6,7} \filldraw (\x,0) circle (1pt);
\foreach \x in {1,2,3,4,5,6,7} \filldraw (\x,2) circle (1pt);
\foreach \x in {1,2,3,4,5,6,7} \filldraw (\x,4) circle (1pt);
\foreach \x in {1,2,3,4,5,6,7} \filldraw (\x,6) circle (1pt);
\foreach \x in {1,2,3,4,5,6,7} \filldraw (\x,8) circle (1pt);
\foreach \x in {1,2,3,4,5,6,7} \filldraw (\x,10) circle (1pt);
\foreach \x in {1,2,3,4,5,6,7} \filldraw (\x,12) circle (1pt);
\foreach \x in {1,2,3,4,5,6,7} \filldraw (\x,14) circle (1pt);
\foreach \x in {1,2,3,4,5,6,7} \filldraw (\x,16) circle (1pt);
\foreach \x in {1,2,3,4,5,6,7} \filldraw (\x,18) circle (1pt);
\draw (4,0) to [bend left=80] (5,0);
\draw (4,2) to [bend right=80] (5,2);
\draw (3,4) to [bend left=80] (4,4);
\draw (3,6) to [bend right=80] (4,6);
\draw (5,2) to [bend left=80] (6,2);
\draw (5,4) to [bend right=80] (6,4);
\draw (2,8) to [bend left=80] (3,8);
\draw (2,10) to [bend right=80] (3,10);
\draw (4,6) to [bend left=80] (5,6);
\draw (4,8) to [bend right=80] (5,8);
\draw (1,12) to [bend left=80] (2,12);
\draw (1,14) to [bend right=80] (2,14);
\draw (3,10) to [bend left=80] (4,10);
\draw (3,12) to [bend right=80] (4,12);
\draw (2,16) to [bend left=80] (3,16);
\draw (2,18) to [bend right=80] (3,18);
\draw (6,14) to [bend left=80] (7,14);
\draw (6,16) to [bend right=80] (7,16);
\draw (1,0) to (1,12);
\draw (2,0) to (2,8);
\draw (3,0) to (3,4);
\draw (6,0) to (6,2);
\draw (7,0) to (7,14);
\draw (1,14) to (1,18);
\draw (2,10) to (2,12);
\draw (2,14) to (2,16);
\draw (3,12) to (3,16);
\draw (3,6) to (3,8);
\draw (4,2) to (4,4);
\draw (4,8) to (4,10);
\draw (4,12) to (4,18);
\draw (5,4) to (5,6);
\draw (5,8) to (5,18);
\draw (6,4) to (6,14);
\draw (6,16) to (6,18);
\draw (7,16) to (7,18);
\end{tikzpicture}}
\caption{}
\label{fig:MainResultE}
\end{subfigure}
\caption{Given a diagram $d$, we can obtain a reduced factorization by constructing the corresponding labeled directed graph $G_d^S$, which yields the canonical representation of the heap that indexes $d$. We have color-coded the corresponding $1$-regions and entries of the heap.}
\label{fig:MainResult}
\end{figure}


\section{Closing Remarks}\label{sec:ClosingRemarks}

If $(W,S)$ is a Coxeter system of type $\Gamma$, the associated Hecke algebra $\H(\Gamma)$ is an algebra with a basis given by $\{T_w\mid w \in W\}$ and relations that deform the relations of $W$ by a parameter $q$.  As mentioned in Section~\ref{sec:Intro}, the ordinary Temperley--Lieb algebra $\TL(A_n)$ is a quotient of the corresponding Hecke algebra $\H(A_n)$. This realization of the Temperley--Lieb algebra as a Hecke algebra quotient was generalized by Graham~\cite{Graham1995} to the case of an arbitrary Coxeter system. In general, the Temperley--Lieb algebra $\TL(\Gamma)$ is a quotient of $\H(\Gamma)$ having several bases indexed by the FC elements of $W$~\cite[Theorem 6.2]{Graham1995}.

When a faithful diagrammatic representation of $\TL(\Gamma)$ is known to exist, multiplication in the diagram algebra is given by applying local combinatorial rules to the diagrams. In each case, one can choose a basis for the diagram algebra so that each basis diagram is indexed by an FC element, where the diagrams indexed by the distinguished generators of the Coxeter group form a set of ``simple diagrams" that generate the algebra. Every factorization of a basis diagram in terms of simple diagrams corresponds precisely to a factorization of the FC element that indexes the diagram. 

Given a reduced expression for an FC element, obtaining the corresponding diagram is straightforward. All one needs to do is concatenate the sequence of simple diagrams determined by the reduced expression and then apply the appropriate local combinatorial rules in the diagram algebra. However, it is another matter to reverse this process.  That is, given a basis diagram, can one obtain a factorization in terms of simple diagrams, or equivalently obtain a reduced expression for the FC element that indexes the diagram?  Theorem~\ref{thm:MainResult} answers this question in the affirmative in the case of type $A_n$.  What happens with the other types where faithful diagrammatic representations are known to exist?  For example, can we find factorization algorithms for the Temperley--Lieb diagram algebras of types $B_n$, $D_n$, $E_n$, $\widetilde{A}_n$, and $\widetilde{C}_n$?


\bibliographystyle{plain}
\bibliography{DiagramFactorization}


\end{document}